\documentclass{aims}
\usepackage{amsmath}
  \usepackage{paralist}
  \usepackage{graphics} 
  \usepackage{epsfig} 
\usepackage{graphicx}  \usepackage{epstopdf}
 \usepackage[colorlinks=true]{hyperref}
\hypersetup{urlcolor=blue, citecolor=red}

\def\currentvolume{X}
 \def\currentissue{X}
  \def\currentyear{200X}
   \def\currentmonth{XX}
    \def\ppages{X--XX}
     \def\DOI{10.3934/xx.xx.xx.xx}

\newtheorem{theorem}{Theorem}[section]

\newtheorem{lemma}[theorem]{Lemma}

\theoremstyle{definition}
\newtheorem{definition}[theorem]{Definition}
\newtheorem{remark}{Remark}

\newcommand{\esssup}{\mathop{\mathrm{esssup}}}
\newcommand{\essinf}{\mathop{\mathrm{essinf}}}

\title[Stochastic Recursive Optimal Control with Time Delay] 
      {Stochastic Recursive Optimal Control Problem with Time Delay and Applications}

\author[Jingtao Shi and Huanshui Zhang]{}


 \email{shijingtao@sdu.edu.cn}
 \email{hszhang@sdu.edu.cn}

\thanks{J. Shi acknowledges the financial support from the National Natural Science Foundations of China (11301011, 11201264) and Shandong Province (ZR2011AQ012). H. Zhang acknowledges the financial support from the National Natural Science Foundation for Distinguished Young Scholars of China (60825304), the National Basic Research Development Program of China (973 Program, No. 2009CB320600), and the National Natural Science Foundation of China (61104050).}

\begin{document}
\maketitle

\centerline{\scshape Jingtao Shi}
\medskip
{\footnotesize
 \centerline{School of Mathematics, Shandong University, Jinan 250100, China}
   \centerline{Qilu Securities Institute of Financial Studies, Shandong University, Jinan 250100, China}
} 

\medskip

\centerline{\scshape Huanshui Zhang}
\medskip
{\footnotesize
 \centerline{School of Control Science and Engineering, Shandong University, Jinan 250061, China}
}

\bigskip

 \centerline{(Communicated by the associate editor name)}

\begin{abstract}
This paper is concerned with a stochastic recursive optimal control problem with time delay, where the controlled system is described by a stochastic differential delayed equation (SDDE) and the cost functional is formulated as the solution to a backward SDDE (BSDDE). When there are only the pointwise and distributed time delays in the state variable, a generalized Hamilton-Jacobi-Bellman (HJB) equation for the value function in finite dimensional space is obtained, applying dynamic programming principle. This generalized HJB equation admits a smooth solution when the coefficients satisfy a particular system of first order partial differential equations (PDEs). A sufficient maximum principle is derived, where the adjoint equation is a forward-backward SDDE (FBSDDE). Under some differentiability assumptions, the relationship between the value function, the adjoint processes and the generalized Hamiltonian function is obtained. A consumption and portfolio optimization problem with recursive utility in the financial market, is discussed to show the applications of our result. Explicit solutions in a finite dimensional space derived by the two different approaches, coincide.
\end{abstract}

\section{Introduction}

The research of many natural and social phenomena shows that the future development of many processes
depends not only on their present state but also essentially on their previous history. Such processes can be described
by stochastic differential delayed equations (SDDEs). Many examples, such as population dynamics models in biology and memory or inertia representation models in finance, can be found in Kolmanovskii and Shaikhet \cite{KS96}, Mohammed \cite {Mo98}. Stochastic optimal control problems with time delay are those whose dynamics of states are described by SDDEs, and to find some optimal control to maximize/minimize the corresponding cost functionals. In general, stochastic optimal control problems with time delay are practically intractable, because of their infinite-dimensional nature.

However, in certain cases which are still very interesting for the applications, stochastic optimal control problems with time delay can be reduced to a finite-dimensional one and solved explicitly. To the best of our knowledge, the first example of such a solvable problem is a linear delayed system with a quadratic cost functional, given by Kolmanovskii and Maizenberg \cite{KM73}, where only the pointwise and distributed time delays are involved in the state variable (see (\ref{definition of delay terms}) in this section). Then Elsanosi et al. \cite{EOS00} consider optimal harvesting of systems described by SDDEs, where the value function of the problem depends on the initial path of the process in a simple way. Maximum principle approach was developed by \O ksendal and Sulem \cite{OS00} for optimal control of stochastic systems with delay, where the adjoint equations are described as three backward SDDEs (BSDDEs) and one of the adjoint processes need to be zero. Dynamic programming principle for optimal control problems of systems described by SDDEs was obtained by Larssen \cite{La02} when both the dynamics and the cost depends on the past in a general way. As applications, systems where the value function depends on the past only through some weighted average were studied. The finite dimensional Hamilton-Jacobi-Bellman (HJB) equation for the value function of such problems was derived by Larssen and Risebro \cite{LR03}, and the solvability of it was guaranteed by a particular system of first order partial differential equations (PDEs). Extensions for stochastic optimal control problems with time delay, to jump diffusions can be seen in \O ksendal et al. \cite{OSZ11} and to infinite horizon were researched by Agram et al. \cite{AHOP13} recently.

The nonlinear backward stochastic differential equation (BSDEs) was first introduced by Pardoux and Peng \cite{PP90}. Independently, Duffie and Epstein \cite{DE92} introduced the BSDE when they presented a stochastic differential formulation of recursive utility. Later, found by El Karoui et al. \cite{EPQ97}, the recursive utility process can be regarded as the solution to some special BSDE. The stochastic recursive optimal control problem is the one whose cost functional is described by the solution to a BSDE. In this setting, the controlled systems become forward-backward stochastic differential equations (FBSDEs). This kind of optimal control problem has important applications in mathematical economics and finance; see Schroder and Skiadas \cite{SS99}, El Karoui et al. \cite{EPQ01}, \O ksendal and Sulem \cite{OS09}, Wang and Wu \cite{WW09}, Shi and Wu \cite{ShiWu10}, Shi and Yu \cite{ShiYu13}.

It is natural to study stochastic recursive optimal control problems or forward-backward stochastic control systems with time delay, by involving time delays of the state and/or the control variables in the coefficients of the state dynamics and/or the cost functionals. In this case, the cost functional is described as the solution to some BSDDE which is a natural generalization of the classical BSDE to time delayed one. To our best knowledge, Fuhrman et al. \cite{FMT10} first considered one special case of forward-backward stochastic control system with time delay under infinite dimensional space framework, and the value function was proved to be a mild solution to the corresponding HJB equation and the existence of optimal controls in the weak sense was given. In Chen and Wu \cite{ChenWu12}, stochastic recursive optimal control problem with time delay in a general form was considered and the dynamic programming principle was presented. The value function was proved to be the viscosity solution to the corresponding infinite dimensional HJB equation. The optimal control problem of an infinite horizon system governed by a forward-backward SDDE (FBSDDE) was studied by Agram and \O ksendal \cite{AO14}. Sufficient and necessary maximum principles for optimal control under partial information were obtained. An optimal consumption problem with respect to recursive utility from a cash flow with delay was discussed. However, since in their paper the adjoint backward equation was described as an anticipated or time-advanced BSDE (ABSDE) of Peng and Yang \cite{PY09}'s type, no explicit solution was given (note that a solvable special case was given only when trivially there was no time delay in \cite{AO14}). We point out that the ABSDE was another important generalization of classical BSDE, which was very useful to represent the adjoint equation when dealing with the stochastic optimal control problem especially with time delay in the control variable; see Chen and Wu \cite{ChenWu10}, Yu \cite{Yu12}. However, it is in general very difficult to find explicit solutions to this kind of ABSDEs when dealing with real-world problems, though some solvable and numerical results have been published in very special cases.

In the present paper, different from all the above literatures, we study the following stochastic recursive optimal control problem with time delay. Let $\{W(t),t\geq0\}$ be a one-dimensional Brownian motion on some probability space $(\Omega,\mathcal{F},\mathbb{P})$. For $s\geq0$, we assume that the completed filtration $\mathcal{F}^s_t=\sigma\{W(\tau);s\leq\tau\leq t\}$ is augmented by all the $\mathbb{P}$-null sets in $\mathcal{F}$. Let $0<T<\infty$ be the fixed time duration and $0\leq\delta<\infty$ be the constant time delay. Denote $C([-\delta,0];\mathbb{R})$ the Banach space of continuous functions $\gamma:[-\delta,0]\rightarrow\mathbb{R}$ with norm $||\gamma||_C:=\sup\limits_{-\delta\leq t\leq0}|\gamma(t)|$.

For given initial time $s\in[0,T)$, we consider the following controlled SDDE
\begin{equation}\label{controlled SDDE}
\left\{
\begin{aligned}
dX^{s,\varphi;u}(t)&=b\big(t,X^{s,\varphi;u}(t),X_1^{s,\varphi;u}(t),X_2^{s,\varphi;u}(t),u(t)\big)dt\\
     &\quad+\sigma\big(t,X^{s,\varphi;u}(t),X_1^{s,\varphi;u}(t),X_2^{s,\varphi;u}(t),u(t)\big)dW(t),\quad t\in(s,T],\\
 X^{s,\varphi;u}(t)&=\varphi(t-s),\ t\in[s-\delta,s].
\end{aligned}
\right.
\end{equation}
Here continuous function $\varphi:[-\delta,0]\rightarrow\mathbb{R}$ is
the initial path of $X^{s,\varphi;u}(\cdot)$. Let
$\mathbb{U}\subset\mathbb{R}$ be a nonempty convex set. Control $u:\Omega\times[0,T]\rightarrow\mathbb{U}$ is an
$\mathcal{F}^s_t$-adapted process and
\begin{equation}\label{definition of delay terms}
\begin{aligned}
X_1^{s,\varphi;u}(t)=\int_{-\delta}^0e^{\lambda\tau}X^{s,\varphi;u}(t+\tau)d\tau,\quad
X_2^{s,\varphi;u}(t)=X^{s,\varphi;u}(t-\delta),
\end{aligned}
\end{equation}
represent given functionals of the path segment
$X^{s,\varphi;u}_t:=\big\{X^{s,\varphi;u}(t+\tau);\tau\in[-\delta,0]\big\}$ of $X^{s,\varphi;u}(\cdot)$.
$\lambda\in\mathbb{R}$ is the averaging parameter.
$b:[0,T]\times\mathbb{R}^3\times\mathbb{U}\rightarrow\mathbb{R}$ and $\sigma:[0,T]\times\mathbb{R}^3\times\mathbb{U}\rightarrow\mathbb{R}$
are given continuous functions.

Next, we introduce the following controlled BSDDE
\begin{equation}\label{controlled BSDDE}
\left\{
\begin{aligned}
-dY^{s,\varphi;u}(t)&=f\big(t,X^{s,\varphi;u}(t),X_1^{s,\varphi;u}(t),X_2^{s,\varphi;u}(t),Y^{s,\varphi;u}(t),Z^{s,\varphi;u}(t),u(t)\big)dt\\
                      &\quad-Z^{s,\varphi;u}(t)dW(t),\quad t\in[s,T],\\
  Y^{s,\varphi;u}(T)&=\phi(X^{s,\varphi;u}(T),X_1^{s,\varphi;u}(T)).
\end{aligned}
\right.
\end{equation}
Here $f:[0,T]\times\mathbb{R}^3\times\mathbb{R}^2\times\mathbb{U}\rightarrow\mathbb{R},\phi:\mathbb{R}^2\rightarrow\mathbb{R}$ are given functions.

For given $(s,\varphi)\in[0,T)\times C([-\delta,0];\mathbb{R})$ and control $u(\cdot)$, the recursive utility functional of our problem is defined as follows
\begin{equation}\label{recursive utility functional}
\begin{aligned}
J(s,\varphi;u(\cdot))=-Y^{s,\varphi;u}(t)\big|_{t=s}.
\end{aligned}
\end{equation}
We define $\mathcal{U}[s,T]$ as the set of admissible controls $u(\cdot)$ such that if
$u(\cdot)\in\mathcal{U}[s,T]$ then the SDDE (\ref{controlled SDDE}) and BSDDE (\ref{controlled BSDDE}) with (\ref{definition of delay terms}) admit
unique $\mathcal{F}^s_t$-adapted solutions $X^{s,\varphi;u}(\cdot)$ and $(Y^{s,\varphi;u}(\cdot),Z^{s,\varphi;u}(\cdot))$, respectively, for given initial data $(s,\varphi)\in[0,T)\times C([-\delta,0];\mathbb{R})$.

{\bf Problem (SROCPD).}\quad The {\it stochastic recursive optimal control problem with time delay} is to find an optimal control $u^*(\cdot)\in\mathcal{U}[s,T]$, such that
\begin{equation}\label{value function for recursive utility}
V(s,\varphi):=J(s,\varphi;u^*(\cdot))=\essinf\limits_{u(\cdot)\in\ \mathcal{U}[s,T]}J(s,\varphi;u(\cdot)),
\end{equation}
for all $(s,\varphi)\in[0,T)\times C([-\delta,0];\mathbb{R})$, with $V(T,\varphi)=-\phi(\varphi),\varphi\in C([-\delta,0];\mathbb{R})$. The corresponding solutions $(X^{s,\varphi;u^*}(\cdot),Y^{s,\varphi;u^*}(\cdot),Z^{s,\varphi;u^*}(\cdot))$ to (\ref{controlled SDDE}), (\ref{definition of delay terms}) and (\ref{controlled BSDDE}) are called the optimal states, $(X^{s,\varphi;u^*}(\cdot),u^*(\cdot))$ is called the optimal pair and $V(s,\varphi)$ is called the value function.

This problem can be reformulated as follows. The state processes triple $(X^{s,\varphi;u}(\cdot),\\Y^{s,\varphi;u}(\cdot),Z^{s,\varphi;u}(\cdot))$ satisfies the following controlled FBSDDE
\begin{equation}\label{controlled FBSDDE}
\left\{
\begin{aligned}
dX^{s,\varphi;u}(t)&=b\big(t,X^{s,\varphi;u}(t),X_1^{s,\varphi;u}(t),X_2^{s,\varphi;u}(t),u(t)\big)dt\\
     &\quad+\sigma\big(t,X^{s,\varphi;u}(t),X_1^{s,\varphi;u}(t),X_2^{s,\varphi;u}(t),u(t)\big)dW(t),\\
-dY^{s,\varphi;u}(t)&=f\big(t,X^{t,\varphi;u}(t),X_1^{t,\varphi;u}(t),X_2^{t,\varphi;u}(t),Y^{s,\varphi;u}(t),Z^{s,\varphi;u}(t),u(t)\big)dt\\
                      &\quad-Z^{t,\varphi;u}(t)dW(t),\quad t\in[s,T],\\
  X^{s,\varphi;u}(t)&=\varphi(t-s),\ t\in[s-\delta,s],\\
  Y^{s,\varphi;u}(T)&=\phi(X^{s,\varphi;u}(T),X_1^{s,\varphi;u}(T)),
\end{aligned}
\right.
\end{equation}
and the cost functional is given of the form
\begin{equation}\label{cost functional for FBSDDE}
\begin{aligned}
 &J(s,\varphi;u(\cdot))=-Y^{s,\varphi;u}(s)\\
=&-\mathbb{E}^{s,\varphi;u}\bigg[\int_s^Tf\big(t,X^{t,\varphi;u}(t),X_1^{t,\varphi;u}(t),X_2^{t,\varphi;u}(t),Y^{s,\varphi;u}(t),Z^{s,\varphi;u}(t),u(t)\big)dt\\
 &\qquad\qquad+\phi(X^{s,\varphi;u}(T),X_1^{s,\varphi;u}(T))\bigg].
\end{aligned}
\end{equation}
Here $\mathbb{E}^{s,\varphi;u}[\cdot]$ denotes expectation with respect to
the law of $X^{s,\varphi;u}(\cdot)$.

{\bf Problem (FBSOCPD).}\  The {\it forward-backward stochastic optimal control problem with time delay} is to find an optimal control $u^*(\cdot)\in\mathcal{U}[s,T]$, such that
\begin{equation}\label{value function for FBSOCPD}
V(s,\varphi):=J(s,\varphi;u^*(\cdot))=\essinf\limits_{u(\cdot)\in\ \mathcal{U}[s,T]}J(s,\varphi;u(\cdot)),
\end{equation}
for all $(s,\varphi)\in[0,T)\times C([-\delta,0];\mathbb{R})$, with $V(T,\varphi)=-\phi(\varphi),\varphi\in C([-\delta,0];\mathbb{R})$.

In general, the solution to {\bf Problem (SROCPD)} (or equivalently, {\bf Problem (FBSOCPD)}) will depend on the initial path $\varphi$, which is in an infinite dimensional space $C([-\delta,0];\mathbb{R})$. As mentioned before, we expect that in some special cases it can be reduced to a finite dimensional one. In such a context, the crucial point is to investigate when this finite dimensional reduction of the problem is possible and/or to find conditions ensuring that. Several papers have made pioneering effort on this topic for stochastic optimal control problems with time delay (not recursive); see \cite{KM73}, \cite{EOS00}, \cite{OS00},  \cite{LR03}. Motivated by this point and its applicable prospect, in this paper we seek the conditions to ensure that {\bf Problem (SROCPD)} can be reduced to a finite dimensional one. Specifically, we show that if the system (\ref{controlled FBSDDE}) is on the form
\begin{equation}\label{controlled FBSDDE-finite dimensional}
\left\{
\begin{aligned}
 dX^{s,\varphi;u}(t)&=\Big[b_1\big(t,X^{s,\varphi;u}(t),X_1^{s,\varphi;u}(t),u(t)\big)\\
                    &\quad\ +b_2\big(t,X^{s,\varphi;u}(t),X_1^{s,\varphi;u}(t),u(t)\big)X_2^{s,\varphi;u}(t)\Big]dt\\
                    &\quad\ +\sigma\big(t,X^{s,\varphi;u}(t),X_1^{s,\varphi;u}(t),u(t)\big)dW(t),\\
-dY^{s,\varphi;u}(t)&=\Big[f_1\big(t,X^{t,\varphi;u}(t),X_1^{t,\varphi;u}(t),Y^{s,\varphi;u}(t),Z^{s,\varphi;u}(t),u(t)\big)\\
                    &\quad\ +f_2\big(t,X^{t,\varphi;u}(t),X_1^{t,\varphi;u}(t),Y^{s,\varphi;u}(t),Z^{s,\varphi;u}(t),u(t)\big)X_2^{s,\varphi;u}(t)\Big]dt\\
                    &\quad-Z^{t,\varphi;u}(t)dW(t),\quad t\in[s,T],\\
  X^{s,\varphi;u}(t)&=\varphi(t-s),\quad t\in[s-\delta,s],\\
  Y^{s,\varphi;u}(T)&=\phi(X^{s,\varphi;u}(T),X_1^{s,\varphi;u}(T)),
\end{aligned}
\right.
\end{equation}
the problem can be reduced to a finite dimensional one and its solvability could be guaranteed, provided an auxiliary system of first order PDEs involving the coefficients $b_1,b_2,\sigma,f_1,f_2$ and $\phi$ admits a solution. Though the main result which we obtained in this paper is for the controlled system (\ref{controlled FBSDDE-finite dimensional}) which is less general than (\ref{controlled FBSDDE}), it never the less covers many interesting applications. This is the first main contribution of this paper, and we will make this point clear in Section 2.

The other main contribution in this paper is that we first study the relationship between Bellman's dynamic programming and Pontryagin's maximum principle approaches, for stochastic recursive optimal control problems with time delay. Such a topic is of great importance in delay-free stochastic control theory; see the systematic monograph by Yong and Zhou \cite{YZ99}. Since the relationship between these two approaches is the one between the derivatives of the value function and the adjoint processes along the optimal state, or actually the one between HJB equations and stochastic Hamiltonian systems, and more generally, the one between PDEs and SDEs. For recent development of the relationship between dynamic programming and maximum principle for stochastic optimal control problems (without delay but including jump diffusions, Markov switching, singular control, or FBSDE systems), refer to Framstad et al. \cite{FOS04}, Shi and Wu \cite{ShiWu11}, Donnelly \cite{Don11}, Zhang et al. \cite{ZES12}, Bahlali et al. \cite{BCM12}, Shi and Yu \cite{ShiYu13}, Chighoub and Mezerdi \cite{CM14}. Thereby, it is natural to ask the question: Are there any relations between these two extensively used and important approaches, for stochastic optimal control problems with time delay? The answer should be yes. However, to our best knowledge, results on this topic are quite lacking in the literature, except the one by the first author \cite{Shi11}. One main difficulty and obstacle is that the solution to a stochastic optimal control problem with time delay, or controlled system with SDDE, will be in an infinite dimensional space framework. Moreover, their solvability in the infinite dimensional spaces is complicated and consequently their real applications are largely limited. Due to the special dependence on the past trajectory via terms $X_1(t)$ and $X_2(t)$ in (\ref{definition of delay terms}), in this paper we first prove a sufficient maximum principle when the terminal condition $\phi$ has some linear form. Note that our result can not be covered by Theorem 3.1 of \cite{AO14}, since they use time advanced FBSDE (AFBSDE) to describe the adjoint processes while we use the FBSDDE. Then we find the connections between the derivatives of the value function and the adjoint processes along the optimal state, assuming that the value function depends on the initial path of the process in a simple way and is smooth enough. The main result is shown in Section 3.

Rich literatures can be found for the financial applications of stochastic optimal control problems with time delay. For example, refer to \cite{KS96}, \cite{Mo98} for population growth models in biology, to \cite{OS00}, \cite{ChenWu10}, \cite{OSZ11} for optimal consumption choice problems, to Gozzi and Marinelli \cite{GM06} for advertising models, to Federico \cite{Federico11} for pension fund models, to Pang et al. \cite{CPY11} for portfolio optimization models and to Arriojas et al. \cite{AHMP07}, Mao and Sabanis \cite{MS13} for option pricing models in financial market. However, because of the infinite dimensional framework in many cases, no explicit solution exists, and numerical solutions are very difficult to obtain. This is one motivation for us to study the controlled system with time delay in the forms of $X_1(t)$ and $X_2(t)$ as defined in (\ref{definition of delay terms}). In Section 4, inspired by the applicable examples in \cite{OS00} and particularly in \cite{CPY11}, a consumption and portfolio optimization problem with recursive utility in the financial market is discussed. Another main contribution in this paper is that we obtain the explicit solution in finite dimensional space for this problem. Via an investigation of the corresponding PDEs system to guarantee the corresponding generalized HJB equation is effective, a complete discussion is possible and the theoretical results obtained in the previous sections are justified.

The rest of this paper is organized as follows. In Section 2, under some suitable assumptions, we investigate that under what conditions on the coefficients, the generalized HJB equation obtained by \cite{ChenWu12} via dynamic programming can be reduced to a finite dimensional one. An the main result is a stochastic verification theorem. In Section 3, after deriving a sufficient maximum principle for the optimal control, we obtain the relationship the two approaches: dynamic programming principle and maximum principle. Under the assumption that the value function is smooth enough, the relations among its derivatives, the adjoint processes and the generalized Hamiltonian function are given. A consumption and portfolio optimization problem with recursive utility in the financial market is discussed in Section 4, to show the applications of our result. Explicit solutions in a finite dimensional space derived by the maximum principle and dynamic programming approaches, coincide. Finally, Section 5 gives some concluding remarks.

\section{Preliminaries and the Generalized HJB Equation in Finite Dimension}

In this section, we focus on the dynamic programming approach for {\bf Problem (SROCPD)}. We first present a stochastic verification theorem, where the generalized HJB equation in Theorem 4.9 of \cite{ChenWu12} is reduced to a finite dimensional one, by assuming that the value function of our problem depends on the initial path of the state process in a simple way. Then we find condition on the coefficients $b_1,b_2,\sigma,f_1,f_2$ and $\phi$ to ensure the above reduction is effective and applicable, which is a system of first order PDEs. The results in this section can be regarded as the extension of those in \cite{LR03} to recursive utility case.

For any $s\in[0,T)$, the following notations are used in this paper.
\begin{equation*}
\begin{aligned}
&L^2(\Omega,\mathcal{F}_T^s;\mathbb{R}):=\Big\{\mathbb{R}\mbox{-valued }\mathcal{F}_T^s\mbox{-measurable random variables }\xi;
 \mbox{ } \mathbb{E}|\xi|^2<\infty\Big\},
\end{aligned}
\end{equation*}
\begin{equation*}
\begin{aligned}
&S^2_\mathcal{F}([s,T];\mathbb{R}):=\Big\{\mathbb{R}\mbox{-valued }\mathcal{F}_t^s\mbox{-adapted processes }\psi(t);
 \mbox{ }\mathbb{E}\Big[\sup\limits_{s\leq t\leq T}|\psi(t)|\Big]<\infty\Big\},
\end{aligned}
\end{equation*}
\begin{equation*}
\begin{aligned}
&L^2_\mathcal{F}([s,T];\mathbb{R}):=\Big\{\mathbb{R}\mbox{-valued }\mathcal{F}_t^s\mbox{-adapted processes }\psi(t);
 \mbox{ }\mathbb{E}\int_s^T|\psi(t)|^2dt<\infty\Big\},
\end{aligned}
\end{equation*}
\begin{equation*}
\begin{aligned}
&\mathbb{L}^2(\Omega;C([-\delta,0];\mathbb{R});\mathcal{F}_t^s):=\Big\{\psi:\Omega\rightarrow C([-\delta,0];\mathbb{R})\mbox{ is }
 \mathcal{F}_t^s\mbox{-measurable};\mbox{ the norm}\\
&\qquad\qquad||\varphi||^2_{\mathbb{L}^2(\Omega;C([-\delta,0];\mathbb{R}))}:=\mathbb{E}\Big[||\varphi(\omega)||^2_{C([-\delta,0];\mathbb{R})}\Big]<\infty\Big\}.
\end{aligned}
\end{equation*}

First we introduce the following assumptions.

\noindent(H1)\ The functions $b(t,x,x_1,x_2,u)$ and
$\sigma(t,x,x_1,x_2,u)$ are joint continuous and globally Lipschitz in $(x,x_1,x_2)$.

\noindent(H2)\ There exists a constant $C>0$ such that
\begin{equation*}
\begin{aligned}
&|b(t,x,x_1,x_2,u)|+|\sigma(t,x,x_1,x_2,u)|\leq C(1+|x|+|x_1|+|x_2|),
\end{aligned}
\end{equation*}
for all $t\in[0,T],x\in\mathbb{R},x_1\in\mathbb{R},x_2\in\mathbb{R}$ and $u\in\mathbb{U}$.

\noindent(H3)\ The initial path $\varphi$ belongs to the space
$\mathbb{L}^2(\Omega;C([-\delta,0];\mathbb{R});\mathcal{F}_s^s)$ of
$\mathcal{F}_s^s$-measurable elements in
$\mathbb{L}^2(\Omega;C([-\delta,0];\mathbb{R}))$, that is,
$\varphi:\Omega\rightarrow C([-\delta,0];\mathbb{R})$ is
$\mathcal{F}_s^s$-measurable and
\begin{equation*}
||\varphi||^2_{\mathbb{L}^2(\Omega;C([-\delta,0];\mathbb{R}))}:=\mathbb{E}\Big[||\varphi(\omega)||^2_{C([-\delta,0];\mathbb{R})}\Big]<\infty.
\end{equation*}

The following classical result can be seen in \cite{Mo98}.

\begin{lemma}
Let assumptions (H1)$\sim$(H3) hold, then for any
$(s,\varphi)\in[0,T)\times\mathbb{L}^2(\Omega;C([-\delta,0];\mathbb{R}))$ and
$u(\cdot)\in\mathcal{U}[s,T]$, the SDDE (\ref{controlled SDDE}) with (\ref{definition of delay terms})
admits a unique adapted strong solution $X^{s,\varphi;u}(\cdot)\in\mathbb{L}^2(\Omega;C([-\delta,0];\mathbb{R});\mathcal{F}_t^s)$.
\end{lemma}

We also need the following assumptions.

\noindent(H4)\quad $f(\cdot,x,x_1,x_2,y,z,u)$ is $\mathcal{F}_t^s$-measurable, for all $x,x_1,x_2\in\mathbb{R}$ and $u\in\mathbb{U}$.

\noindent(H5)\quad The functions $f(t,x,x_1,x_2,y,z,u)$ and
$\phi(x,x_1)$ are joint continuous and globally Lipschitz in $(x,x_1,x_2,y,z)$.

\noindent(H6)\quad There exists a constant $C>0$ such that
\begin{equation*}
\begin{aligned}
&|f(t,x,x_1,x_2,0,0,u)|+|\phi(x,x_1)|\leq C(1+|x|+|x_1|+|x_2|),
\end{aligned}
\end{equation*}
for all $t\in[0,T],x,x_1,x_2\in\mathbb{R}$ and $u\in\mathbb{U}$.

The following result can be obtained from the classical BSDE theory, by Lemma 2.1. See also \cite{ChenWu12} in detail.

\begin{lemma}
Let assumptions (H1)$\sim$(H6) hold, then for any
$(s,\varphi)\in[0,T)\times\mathbb{L}^2(\Omega;C([-\delta,0];\mathbb{R}))$ and
$u(\cdot)\in\mathcal{U}[s,T]$, the BSDDE (\ref{controlled BSDDE}) with (\ref{definition of delay terms})
admits a unique adapted solution $(Y^{s,\varphi;u}(\cdot),Z^{s,\varphi;u}(\cdot))\in S^2_\mathcal{F}([s,T];\mathbb{R})\times L^2_\mathcal{F}([s,T];\mathbb{R})$.
\end{lemma}

We now introduce some preliminaries in infinite dimension, which is also used in \cite{ChenWu12,FMT10,Mo98}. Let $C_b$ be the Banach space of all bounded uniformly continuous functions $\Phi:C([-\delta,0];\mathbb{R})\rightarrow\mathbb{R}$ with the sup norm
\begin{equation*}
||\Phi||_{C_b}:=\sup\limits_{\varphi\in C([-\delta,0];\mathbb{R})}|\Phi(\varphi)|,\ \Phi\in C_b.
\end{equation*}
Define the operator $P_t:C_b\rightarrow C_b,t\geq0$ on $C_b$ by
\begin{equation*}
P_t(\Phi)(\varphi):=\mathbb{E}\big[\Phi\big(X^{0,\varphi;u}(t))\big],\ t\geq0,\ \Phi\in C_b,\ \varphi\in C([-\delta,0];\mathbb{R}).
\end{equation*}
We also define an operator $\mathcal{A}:D(\mathcal{A})\subset C_b\rightarrow C_b$ of $\{P_t\}_{t\geq0}$ by the weak limit
\begin{equation*}
\mathcal{A}(\Phi)(\varphi):=w\mbox{ -}\lim\limits_{t\rightarrow0+}\frac{P_t(\Phi)(\varphi)-\Phi(\varphi)}{t},
\end{equation*}
here $\Phi$ belongs to the domain $D(\mathcal{A})$ of $\mathcal{A}$ if and only if the above weak limit exists in $C_b$. Then we can obtain easily that (\cite{Mo98})
\begin{equation*}
\frac{d}{dt}P_t(\Phi)=\mathcal{A}(P_t(\Phi))=P_t(\mathcal{A}(\Phi)),\ t\geq0,
\end{equation*}
for any $\Phi\in D(\mathcal{A})$.

Let $F_n:=\{\kappa1_{\{0\}}:\kappa\in\mathbb{R}\}$ and $C\oplus F_n:=\{\varphi+\kappa1_{\{0\}}:\varphi\in C([-\delta,0];\mathbb{R}),\kappa\in\mathbb{R}\}$ with norm $||\varphi+\kappa1_{\{0\}}||:=||\varphi||_C+|\kappa|$ for $\kappa\in\mathbb{R}$, where $1_{\{0\}}:[-\delta,0]\rightarrow\mathbb{R}$ is defined by
\begin{equation*}
1_{\{0\}}(\theta):=\left\{\begin{array}{l}
                   0,\mbox{ for }\theta\in[-\delta,0),\\
                   1,\mbox{ for }\theta=0.
                   \end{array}\right.
\end{equation*}
For a Borel measurable function $\Phi:C([-\delta,0];\mathbb{R})\rightarrow\mathbb{R}$, we also define
\begin{equation*}
\mathcal{S}(\Phi)(\varphi):=\lim\limits_{t\rightarrow0}\frac{\Phi(\tilde{\varphi}_t)-\Phi(\varphi)}{t},
\end{equation*}
for all $\varphi\in C([-\delta,0];\mathbb{R})$, where $\tilde{\varphi}(\cdot):[-\delta,T]\rightarrow\mathbb{R}$ is an extension of $\varphi$ defined by
\begin{equation*}
\tilde{\varphi}(t):=\left\{\begin{array}{l}
                         \varphi(t),\mbox{ }t\in[-\delta,0),\\
                         \varphi(0),\mbox{ }t\geq0,
                    \end{array}\right.
\end{equation*}
and $\tilde{\varphi}_t\in C([-\delta,0];\mathbb{R})$ is defined by $\tilde{\varphi}_t(\tau):=\big\{\tilde{\varphi}(t+\tau);\tau\in[-\delta,0]\big\}$. Let $\hat{D}(\mathcal{S})$, the domain of $\mathcal{S}$, the set of $\Phi:C([-\delta,0];\mathbb{R})\rightarrow\mathbb{R}$ such that the above limit exists for each $\varphi\in C([-\delta,0];\mathbb{R})$. Define $D(\mathcal{S})$ as the set of all functions $\Phi:[0,T]\times C([-\delta,0];\mathbb{R})\rightarrow\mathbb{R}$ such that $\Phi(t,\cdot)\in\hat{D}(\mathcal{S}),t\in[0,T]$.

In addition, for each sufficiently smooth function $\Phi$, we denote its first and second Fr\'{e}chet derivatives with respect to $\varphi\in C([-\delta,0];\mathbb{R})$ by $D\Phi$ and $D^2\Phi$. And let $C^{1,2}_{lip}([0,T]\times C)$ be the set of functions $\Phi:[0,T]\times C([-\delta,0];\mathbb{R})\rightarrow\mathbb{R}$ such that $\frac{\partial\Phi}{\partial t},D\Phi,D^2\Phi$ exist and they are globally bounded and Lipschitz continuous.

Then we have the following formula for the generator $\mathcal{A}$, which is a slight modification of Theorem 4.2 in \cite{ChenWu12}, which is also can be seen in \cite{Mo98}.

\begin{lemma}
Suppose that $\Phi\in C^{1,2}_{lip}([0,T]\times C)\cap D(\mathcal{S})$. Let $u(\cdot)\in\mathcal{U}[s,T]$, and $\{X^{s,\varphi;u}(t),t\in[s,T]\}$ be the
Markov solution process to the SDDE (\ref{controlled SDDE}) and (\ref{definition of delay terms}), with the initial data $(s,\varphi)\in[0,T)\times C([-\delta,0];\mathbb{R})$. Then $\Phi\in D(\mathcal{A})$ and for each $\varphi\in C([-\delta,0];\mathbb{R})$, we have
\begin{equation}\label{operator}
\begin{aligned}
\mathcal{A}(\Phi)(\varphi)=\mathcal{S}(\Phi)(\varphi)+\overline{D\Phi(\varphi)}\big(b(s,\varphi,u(\cdot))1_{\{0\}}\big)
                          +\frac{1}{2}\overline{D^2\Phi(\varphi)}\big(\sigma^2(s,\varphi,u(\cdot))1_{\{0\}}\big),
\end{aligned}
\end{equation}
where $\overline{D\Phi(\varphi)}:C\oplus F_n\rightarrow\mathbb{R},\ \overline{D^2\Phi(\varphi)}:(C\oplus F_n)\times(C\oplus F_n)\rightarrow\mathbb{R}$ are the continuous linear and bilinear extensions of $D\Phi(\varphi),D^2\Phi(\varphi)$, respectively.
\end{lemma}

Now, we turn to consider {\bf Problem (SROCPD)} by Belmann's dynamic programming. In general, the value function $V(s,\varphi)$ defined in (\ref{value function for FBSOCPD}) may depend on the initial path $\varphi\in
C([-\delta,0];\mathbb{R})$ in a complicated way. From Theorem 3.7 of \cite{ChenWu12}, we know that the value function satisfies the following generalized dynamic programming principle (DPP):
\begin{equation}\label{generalized dynamic programming principle by Chen and Wu}
\begin{aligned}
V(s,\varphi)&=-\esssup\limits_{u\in\mathcal{U}[s,T]}\mathbb{E}^{s,\varphi;u}
               \Big[\int_s^{\hat{s}}f\big(t,X^{s,\varphi;u}(t),X_1^{s,\varphi;u}(t),X_2^{s,\varphi;u}(t),Y^{s,\varphi;u}(t),\\
              &\qquad\qquad\qquad\qquad\qquad Z^{s,\varphi;u}(t),u(t)\big)dt+V\big(\hat{s},X^{s,\varphi;u}_{\hat{s}}\big)\Big],\ \forall\hat{s}\in[s,T].
\end{aligned}
\end{equation}
Here $X^{s,\varphi;u}_{\hat{s}}$ is the map $X^{s,\varphi;u}_{\hat{s}}:[-\delta,0]\rightarrow\mathbb{R}$ defined by $X^{s,\varphi;u}_{\hat{s}}(\tau):=X^{s,\varphi;u}(\hat{s}+\tau)$.

The following result is an immediate corollary of Theorem 4.9 in \cite{ChenWu12}.

\begin{theorem}
{\bf (Infinite Dimensional Generalized HJB Equation)} Assume that the value function $V(s,\varphi)\in C^{1,2}_{lip}([0,T]\times C)\cap D(\mathcal{S})$ for {\bf Problem (SROCPD)}, then $V(s,\varphi)$ solves the following PDE
\begin{equation}\label{BSEE}
\left\{
\begin{aligned}
&\frac{\partial V}{\partial s}(s,\varphi)+\inf\limits_{u\in\mathbb{U}}\Big\{\mathcal{A}^uV(s,\varphi)
 +\tilde{f}\big(s,\varphi,-V(s,\varphi),-\nabla_0V(s,\varphi)\tilde{\sigma}(s,\varphi,u),u\big)\Big\}\\
&=0,\hspace{32mm}\forall(s,\varphi)\in[0,T)\times C([-\delta,0];\mathbb{R}),\\
&V(T,\varphi)=-\phi(\varphi),\quad \forall \varphi\in C([-\delta,0];\mathbb{R}),
\end{aligned}
\right.
\end{equation}
where $\nabla_0V(s,\varphi)\equiv\nabla_\varphi V(s,\varphi)(\{0\}),\nabla_\varphi V(s,\varphi)$ is the gradient of $V$ with respect to $\varphi$ at point $(s,\varphi)$, and
\begin{equation*}
\begin{aligned}
&\tilde{\sigma}(s,\varphi,u):=\sigma(s,x(\varphi),x_1(\varphi),x_2(\varphi),u),\\ &\tilde{f}(s,\varphi,-V(s,\varphi),-\nabla_0V(s,\varphi)\tilde{\sigma}(s,\varphi,u),u)\\
&\quad:=f\big(s,x(\varphi),x_1(\varphi),x_2(\varphi),-V(s,\varphi),-\nabla_0V(s,\varphi)\tilde{\sigma}(s,\varphi,u),u\big),
\end{aligned}
\end{equation*}
with
\begin{equation*}
x=x(\varphi):=\varphi(0),\ x_1=x_1(\varphi):=\int_{-\delta}^0e^{\lambda\tau}\varphi(\tau)d\tau,\ x_2=x_2(\varphi):=\varphi(-\delta).
\end{equation*}
\end{theorem}

Note that (\ref{BSEE}) is a PDE with terminal condition in the infinite dimensional space. One of the main target in this section is to find out that, under what conditions it can be reduced to a finite dimensional one. Inspired by \cite{LR03}, one might expect that the value function $V(s,\varphi)$ depends on $\varphi$ only through the first two functionals $x(\varphi),x_1(\varphi)$, that is,
\begin{equation}\label{special cost functional}
V(s,\varphi)=V(s,x(\varphi),x_1(\varphi))=V(s,x,x_1),
\end{equation}
and is independent of the third functional $x_2(\varphi)$. If this is the case, the operator $\mathcal{A}$ in (\ref{operator}) is a differential operator and the equation (\ref{BSEE}) is a second order PDE in the finite dimensional space. For this, we first need the following delayed It\^{o}'s formula, whose proof can be seen in \cite{EOS00}.

\begin{lemma}
{\bf (Delayed It\^{o}'s Formula)} Let $g\in C^{1,2,1}([0,T]\times\mathbb{R}^2)$, processes $(X(\cdot),X_1(\cdot),X_2(\cdot))$ are defined by (\ref{controlled SDDE}) and (\ref{definition of delay terms}), then for given $u\in\mathbb{U}$, we have
\begin{equation}\label{Ito's formula}
\begin{aligned}
&dg(t,X(t),X_1(t))\\
&=\bigg\{\frac{\partial g}{\partial t}(t,X(t),X_1(t))+b(t,X(t),X_1(t),X_2(t),u)\frac{\partial g}{\partial x}(t,X(t),X_1(t))\\
&\qquad+\frac{1}{2}\sigma^2(t,X(t),X_1(t),X_2(t),u)\frac{\partial^2 g}{\partial x^2}(t,X(t),X_1(t))\\
&\qquad+\big[X(t)-\lambda X_1(t)-e^{-\lambda\delta}X_2(t)\big]\frac{\partial g}{\partial x_1}(t,X(t),X_2(t))\bigg\}dt\\
&\quad+\sigma(t,X(t),X_1(t),X_2(t),u)\frac{\partial g}{\partial x}(t,X(t),X_1(t))dW(t).
\end{aligned}
\end{equation}
\end{lemma}

The following theorem takes the independence of $x_2$ for the value function $V$ as an assumption, and states a stochastic verification theorem via the finite dimensional counterpart to (\ref{BSEE}).
\begin{theorem}
{\bf (Stochastic Verification Theorem)}
Suppose that the following PDE
\begin{equation}\label{generalized HJB equation}
\left\{
\begin{aligned}
&-\frac{\partial V}{\partial s}(s,x,x_1)+\sup\limits_{u\in\mathbf{U}}G\Big(s,x,x_1,x_2,u,-V(s,x,x_1),-\frac{\partial V}{\partial x}(s,x,x_1),\\
&\qquad\qquad\qquad\qquad\qquad-\frac{\partial^2 V}{\partial x^2}(s,x,x_1),-\frac{\partial V}{\partial x_1}(s,x,x_1)\Big)=0,\ \forall x_2\in\mathbb{R},\\
&V(T,x,x_1)=-\phi(x,x_1),
\end{aligned}
\right.
\end{equation}
admits a sufficiently smooth solution $V$ depends on $(s,x,x_1)$ only and $V(s,x,x_1)\in C^{1,2,1}([0,T]\times\mathbb{R}^2)$, where the generalized Hamiltonian function
$G:[0,T]\times\mathbb{R}^3\times\mathbb{U}\times\mathbb{R}^4\rightarrow\mathbb{R}$
is defined as
\begin{equation}\label{generalized Hamiltonian function}
\begin{aligned}
&G(s,x,x_1,x_2,u,k,p,R,q):=b(s,x,x_1,x_2,u)p+\frac{1}{2}\sigma^2(s,x,x_1,x_2,u)R\\
&\quad+\big(x-\lambda x_1-e^{-\lambda\delta}x_2\big)q+f(s,x,x_1,x_2,k,\sigma(s,x,x_1,x_2,u)p,u).
\end{aligned}
\end{equation}
Then
\begin{equation}\label{stochastic verification theorem-1}
V(s,x,x_1)\leq J(s,x,x_1;u(\cdot)),\quad\forall u(\cdot)\in\mathcal{U}[s,T],\ (s,x,x_1)\in[0,T]\times\mathbb{R}^2.
\end{equation}
Furthermore, an admissible pair $(X^*(\cdot),u^*(\cdot))\equiv(X^{s,x,x_1;u^*}(\cdot),u^*(\cdot))$ is an optimal pair for {\bf Problem (SROCPD)} if and only if
\begin{equation}\label{stochastic verification theorem-2}
\begin{aligned}
&\ G\Big(t,X^*(t),X_1^*(t),X_2^*(t),u^*(t),-V(t,X^*(t),X_1^*(t)),-\frac{\partial V}{\partial x}(t,X^*(t),X_1^*(t)),\\
&\quad\ -\frac{\partial^2V}{\partial x^2}(t,X^*(t),X_1^*(t)),-\frac{\partial V}{\partial x_1}(t,X^*(t),X_1^*(t))\Big)\\
=&\ \max\limits_{u\in\mathbb{U}}G\Big(t,X^*(t),X_1^*(t),X_2^*(t),u,-V(t,X^*(t),X_1^*(t)),-\frac{\partial V}{\partial x}(t,X^*(t),X_1^*(t)),\\
&\qquad\quad-\frac{\partial^2V}{\partial x^2}(t,X^*(t),X_1^*(t)),-\frac{\partial V}{\partial x_1}(t,X^*(t),X_1^*(t))\Big),\ a.e.t\in[s,T],a.s.,
\end{aligned}
\end{equation}
where $X_1^*(t)\equiv X_1^{s,x,x_1;u^*}(t),X_2^*(t)\equiv X_2^{s,x,x_1;u^*}(t)$ are defined as (\ref{definition of delay terms}).
\end{theorem}
\begin{proof}
For any $u(\cdot)\in\mathcal{U}[s,T]$ with the corresponding state $X^u(\cdot)\equiv X^{s,x,x_1;u}(\cdot)$ and $X_1^u(\cdot)\equiv X_1^{s,x,x_1;u}(\cdot),X_2^u(\cdot)\equiv X_2^{s,x,x_1;u}(\cdot)$ defined as (\ref{definition of delay terms}), applying delayed It\^{o}'s formula (\ref{Ito's formula}) to $V(t,X^u(t),X_1^u(t))$, we obtain that
\begin{equation*}
\begin{aligned}
V(s,x,x_1)&=-\mathbb{E}\phi\big(X^u(T),X_1^u(T)\big)-\mathbb{E}\int_s^T\bigg\{\frac{\partial V}{\partial t}(t,X^u(t),X_1^u(t))\\
&\quad+b(t,X^u(t),X_1^u(t),X_2^u(t),u(t))\frac{\partial V}{\partial x}(t,X^u(t),X_1^u(t))\\
&\quad+\frac{1}{2}\big(\sigma(t,X^u(t),X_1^u(t),X_2^u(t),u(t))\big)^2\frac{\partial^2 V}{\partial x^2}(t,X^u(t),X_1^u(t))\\
&\quad+\Big[X^{s,x,x_1;u}(t)-\lambda X_1^u(t)-e^{-\lambda\delta}X_2^u(t)\Big]
 \frac{\partial V}{\partial x_1}(t,X^u(t),X_1^u(t))\bigg\}dt
\end{aligned}
\end{equation*}
\begin{equation*}
\begin{aligned}&=-\mathbb{E}\phi\big(X^u(T),X_1^u(T)\big)
 +\mathbb{E}\int_s^T\bigg\{-\frac{\partial V}{\partial t}(t,X^u(t),X_1^u(t))\\
&\qquad+G\Big(t,X^u(t),X_1^u(t),X_2^u(t),u(t),-V(t,X^u(t),X_1^u(t)),\\
&\qquad\quad-\frac{\partial V}{\partial x}(t,X^u(t),X_1^u(t)),
 -\frac{\partial^2V}{\partial x^2}(t,X^u(t),X_1^u(t)),-\frac{\partial V}{\partial x_1}(t,X^u(t),X_1^u(t))\Big)\\
&\qquad-f\Big(t,X^u(t),X_1^u(t),X_2^u(t),-V(t,X^u(t),X_1^u(t)),\\
&\qquad\quad-\sigma(t,X^u(t),X_1^u(t),X_2^u(t),u(t))\frac{\partial V}{\partial x}(t,X^u(t),X_1^u(t),u(t)\Big)\bigg\}dt\\
&=\ J(s,x,x_1;u(\cdot))+\mathbb{E}\int_s^T\Big\{-\frac{\partial V}{\partial t}(t,X^u(t),X_1^u(t))\\
&\qquad+G\Big(t,X^u(t),X_1^u(t),X_2^u(t),u(t),-V(t,X^u(t),X_1^u(t)),\\
&\qquad\quad-\frac{\partial V}{\partial x}(t,X^u(t),X_1^u(t)),
 -\frac{\partial^2V}{\partial x^2}(t,X^u(t),X_1^u(t)),-\frac{\partial V}{\partial x_1}(t,X^u(t),X_1^u(t))\Big)\Big\}ds\\
&\leq\ J(s,x,x_1;u(\cdot))+\mathbb{E}\int_s^T\Big\{-\frac{\partial V}{\partial t}(t,X^u(t),X_1^u(t))\\
&\qquad+\max\limits_{u\in\mathbb{U}}G\Big(t,X^u(t),X_1^u(t),X_2^u(t),u(t),-V(t,X^u(t),X_1^u(t)),\\
&\qquad\quad-\frac{\partial V}{\partial x}(t,X^u(t),X_1^u(t)),-\frac{\partial^2V}{\partial x^2}(t,X^u(t),X_1^u(t)),
 -\frac{\partial V}{\partial x_1}(t,X^u(t),X_1^u(t))\Big)\Big\}ds\\
&=\ J(s,x,x_1;u(\cdot)).
\end{aligned}
\end{equation*}
The third ``=" in the above holds by the uniqueness of the solution to the BSDDE (\ref{controlled BSDDE}). Thus (\ref{stochastic verification theorem-1}) holds.
Next, applying the above inequality to $(X^*(\cdot),u^*(\cdot))$, we have
\begin{equation*}
\begin{aligned}
&V(s,x,x_1)=\ J(s,x,x_1;u^*(\cdot))+\mathbb{E}\int_s^T\Big\{-\frac{\partial V}{\partial t}(t,X^*(t),X_1^*(t))\\
&\qquad+G\Big(t,X^*(t),X_1^*(t),X_2^*(t),u^*(t),-V(t,X^*(t),X_1^*(t)),\\
&\qquad\quad-\frac{\partial V}{\partial x}(t,X^*(t),X_1^*(t)),-\frac{\partial^2V}{\partial x^2}(t,X^*(t),X_1^*(t)),
 -\frac{\partial V}{\partial x_1}(t,X^*(t),X_1^*(t))\Big)\Big\}ds.
\end{aligned}
\end{equation*}
The desired result follows immediately from the fact that
\begin{equation*}
\begin{aligned}
&-\frac{\partial V}{\partial t}(t,X^*(t),X_1^*(t))+G\Big(t,X^*(t),X_1^*(t),X_2^*(t),u^*(t),-V(t,X^*(t),X_1^*(t)),\\
&\quad-\frac{\partial V}{\partial x}(t,X^*(t),X_1^*(t),-\frac{\partial^2V}{\partial x^2}(t,X^*(t),X_1^*(t)),
 -\frac{\partial V}{\partial x_1}(t,X^*(t),X_1^*(t))\Big)\leq0,
\end{aligned}
\end{equation*}
which is due to PDE (\ref{generalized HJB equation}). The proof is complete.
\end{proof}

PDE (\ref{generalized HJB equation}) is the finite dimensional counterpart of (\ref{BSEE}). However, since the coefficients $b,\sigma$ of the SDDE (\ref{controlled SDDE}) enter into the delayed It\^{o}'s formula, and the generator $f$ of the BSDDE (\ref{controlled BSDDE}) depends on $x_2$, the coefficients of the BSEE (\ref{BSEE}) also depend on $x_2$. Consequently, we cannot apriori expect (\ref{BSEE}) to have solutions independent of $x_2$.

In the sequel, we will clarify that under some conditions on the coefficients $b,\sigma,f$, (\ref{BSEE}) have a solution depending only on $(s,x,x_1)$. In other words, we seek conditions ensuring that a solution to (\ref{BSEE}) will be independent of $x_2$, thus the generalized HJB equation in finite dimension (\ref{generalized HJB equation}) is ``effective". The following theorem is our main result.

\begin{theorem}
The generalized HJB equation in finite dimension (\ref{generalized HJB equation}) admits a smooth function $V=V(s,x,x_1)$ which is independent of $x_2$, if \begin{equation}\label{assumption of b,sigma,f}
\begin{aligned}
&b(t,x,x_1,x_2,u)=b_1(t,x,x_1,u)+b_2(t,x,x_1,u)x_2,\\
&\sigma(t,x,x_1,x_2,u)=\sigma(t,x,x_1,u),\\
&f(t,x,x_1,x_2,y,z,u)=f_1(t,x,x_1,y,z,u)+f_2(t,x,x_1,y,z,u)x_2,
\end{aligned}
\end{equation}
and the following system of first order PDEs
\begin{equation}\label{system of first order PDEs}
\left\{
\begin{aligned}
&\frac{\partial\hat{b}}{\partial x_1}(s,x,x_1,u)+e^{\lambda\delta}\bigg[f_2(s,x,x_1,y,z,u)
-b_2(s,x,x_1,u)\frac{\partial\hat{b}}{\partial x}(s,x,x_1,u)\bigg]=0,\\
&\frac{\partial\sigma}{\partial x_1}(s,x,x_1,u)+e^{\lambda\delta}\bigg[f_2(s,x,x_1,y,z,u)
-b_2(s,x,x_1,u)\frac{\partial\sigma}{\partial x}(s,x,x_1,u)\bigg]=0,\\
&\frac{\partial f_1}{\partial x_1}\Big(s,x,x_1,y,z,u\Big)+e^{\lambda\delta}\bigg[f_2(s,x,x_1,y,z,u)\\
&\hspace{4cm}-b_2(s,x,x_1,u)\frac{\partial f_1}{\partial x}(s,x,x_1,y,z,u)\bigg]=0,\\
&\frac{\partial\phi}{\partial x_1}(x,x_1) +e^{\lambda\delta}\bigg[f_2(s,x,x_1,y,z,u)
-b_2(s,x,x_1,u)\frac{\partial\phi}{\partial x}(x,x_1)\bigg]=0,
\end{aligned}
\right.
\end{equation}
holds for all $(s,x,x_1,y,z)$, where
\begin{equation*}
\hat{b}(s,x,x_1,u)):=b_1(s,x,x_1,u)+e^{\lambda\delta}(x-\lambda x_1)b_2(s,x,x_1,u).
\end{equation*}
\end{theorem}
\begin{proof}
We first know that if $V=V(s,x,x_1)$, then from (\ref{generalized HJB equation}), (\ref{generalized Hamiltonian function}), $V$ satisfies the following generalized HJB equation
\begin{equation}\label{generalized HJB equation-finite dimensional}
\begin{aligned}
&-\frac{\partial V}{\partial s}(s,x,x_1)-\big(x-\lambda x_1-e^{-\lambda\delta}x_2\big)\frac{\partial V}{\partial x_1}(s,x,x_1)\\
&-b(s,x,x_1,x_2,u^*)\frac{\partial V}{\partial x}(s,x,x_1)-\frac{1}{2}\sigma^2(s,x,x_1,x_2,u^*)\frac{\partial^2V}{\partial x^2}(s,x,x_1)\\
&+f\Big(s,x,x_1,x_2,-V(s,x,x_1),-\sigma(s,x,x_1,x_2,u^*)\frac{\partial V}{\partial x}(s,x,x_1),u^*\Big)=0,\ \forall x_2\in\mathbb{R},
\end{aligned}
\end{equation}
with terminal condition $V(T,x,x_1)=-\phi(x,x_1)$. Differentiating (\ref{generalized HJB equation-finite dimensional}) with respect to $x_2$, we obtain
\begin{equation}\label{differentiating HJB equation with respect to x2}
\begin{aligned}
&\ \frac{\partial V}{\partial x_1}(s,x,x_1)=-e^{-\lambda\delta}\bigg\{-\frac{\partial b}{\partial x_2}(s,x,x_1,x_2,u^*)\frac{\partial V}{\partial x}(s,x,x_1)\\
&\quad-\frac{1}{2}\sigma(s,x,x_1,x_2,u^*)\frac{\partial\sigma}{\partial x_2}(s,x,x_1,x_2,u^*)\frac{\partial^2V}{\partial x^2}(s,x,x_1)\\
&\quad+\frac{\partial f}{\partial x_2}\Big(s,x,x_1,x_2,-V(s,x,x_1),-\sigma(s,x,x_1,x_2,u^*)\frac{\partial V}{\partial x}(s,x,x_1),u^*\Big)\\
&\quad-\frac{\partial f}{\partial z}\Big(s,x,x_1,x_2,-V(s,x,x_1),-\sigma(s,x,x_1,x_2,u^*)\frac{\partial V}{\partial x}(s,x,x_1),u^*\Big)\\
&\qquad\times\frac{\partial\sigma}{\partial x_2}(s,x,x_1,x_2,u^*)\frac{\partial V}{\partial x}(s,x,x_1)\bigg\}.
\end{aligned}
\end{equation}
Inserting this into (\ref{generalized HJB equation-finite dimensional}), it takes the form
\begin{equation}\label{generalized HJB equation-finite dimensional-2}
\begin{aligned}
&-\frac{\partial V}{\partial s}(s,x,x_1)+\big(x-\lambda x_1-e^{\lambda\delta}x_2\big)e^{\lambda\delta}\bigg\{-\frac{\partial b}{\partial x_2}(s,x,x_1,x_2,u^*)\frac{\partial V}{\partial x}(s,x,x_1)\\
&\quad-\frac{1}{2}\sigma(s,x,x_1,x_2,u^*)\frac{\partial\sigma}{\partial x_2}(s,x,x_1,x_2,u^*)\frac{\partial^2V}{\partial x^2}(s,x,x_1)\\
&\quad+\frac{\partial f}{\partial x_2}\Big(s,x,x_1,x_2,-V(s,x,x_1),-\sigma(s,x,x_1,x_2,u^*)\frac{\partial V}{\partial x}(s,x,x_1),u^*\Big)\\
&\quad-\frac{\partial f}{\partial z}\Big(s,x,x_1,x_2,-V(s,x,x_1),-\sigma(s,x,x_1,x_2,u^*)\frac{\partial V}{\partial x}(s,x,x_1),u^*\Big)\\
&\qquad\times\frac{\partial\sigma}{\partial x_2}(s,x,x_1,x_2,u^*)\frac{\partial V}{\partial x}(s,x,x_1)\bigg\}\\
&-b(s,x,x_1,x_2,u^*)\frac{\partial V}{\partial x}(s,x,x_1)-\frac{1}{2}\sigma^2(s,x,x_1,x_2,u^*)\frac{\partial^2V}{\partial x^2}(s,x,x_1)\\
&+f\Big(s,x,x_1,x_2,-V(s,x,x_1),-\sigma(s,x,x_1,x_2,u^*)\frac{\partial V}{\partial x}(s,x,x_1),u^*\Big)=0,\ \forall x_2\in\mathbb{R}.
\end{aligned}
\end{equation}

Suppose that $f$ takes the form in (\ref{assumption of b,sigma,f}), then (\ref{generalized HJB equation-finite dimensional-2}) reduces to
\begin{equation}\label{generalized HJB equation-finite dimensional-3}
\begin{aligned}
&-\frac{\partial V}{\partial s}(s,x,x_1)+\big(x-\lambda x_1-e^{-\lambda\delta}x_2\big)e^{\lambda\delta}\bigg\{-\frac{\partial b}{\partial x_2}(s,x,x_1,x_2,u^*)\frac{\partial V}{\partial x}(s,x,x_1)\\
&\quad-\frac{1}{2}\sigma(s,x,x_1,x_2,u^*)\frac{\partial\sigma}{\partial x_2}(s,x,x_1,x_2,u^*)\frac{\partial^2V}{\partial x^2}(s,x,x_1)\\
&\quad-\bigg[\frac{\partial f_1}{\partial z}\Big(s,x,x_1,-V(s,x,x_1),-\sigma(s,x,x_1,x_2,u^*)\frac{\partial V}{\partial x}(s,x,x_1),u^*\Big)\\
&\qquad+\frac{\partial f_2}{\partial z}\Big(s,x,x_1,-V(s,x,x_1),-\sigma(s,x,x_1,x_2,u^*)\frac{\partial V}{\partial x}(s,x,x_1),u^*\Big)x_2\bigg]\\
&\quad\times\frac{\partial\sigma}{\partial x_2}(s,x,x_1,x_2,u^*)\frac{\partial V}{\partial x}(s,x,x_1)\bigg\}\\
&-b(s,x,x_1,x_2,u^*)\frac{\partial V}{\partial x}(s,x,x_1)-\frac{1}{2}\sigma^2(s,x,x_1,x_2,u^*)\frac{\partial^2V}{\partial x^2}(s,x,x_1)\\
&+f_1\Big(s,x,x_1,-V(s,x,x_1),-\sigma(s,x,x_1,x_2,u^*)\frac{\partial V}{\partial x}(s,x,x_1),u^*\Big)=0,\ \forall x_2\in\mathbb{R}.
\end{aligned}
\end{equation}
Next, suppose that $\sigma$ takes the form in (\ref{assumption of b,sigma,f}), then (\ref{generalized HJB equation-finite dimensional-3}) reduces to
\begin{equation}\label{generalized HJB equation-finite dimensional-4}
\begin{aligned}
&-\frac{\partial V}{\partial s}(s,x,x_1)+\big(x-\lambda x_1-e^{\lambda\delta}x_2\big)e^{\lambda\delta}\bigg\{-\frac{\partial b}{\partial x_2}(s,x,x_1,x_2,u^*)\frac{\partial V}{\partial x}(s,x,x_1)\\
&-b(s,x,x_1,x_2,u^*)\frac{\partial V}{\partial x}(s,x,x_1)-\frac{1}{2}\sigma^2(s,x,x_1,u^*)\frac{\partial^2V}{\partial x^2}(s,x,x_1)\\
&+f_1\Big(s,x,x_1,-V(s,x,x_1),-\sigma(s,x,x_1,u^*)\frac{\partial V}{\partial x}(s,x,x_1),u^*\Big)=0,\ \forall x_2\in\mathbb{R}.
\end{aligned}
\end{equation}
Finally, suppose that $b$ takes the form in (\ref{assumption of b,sigma,f}), then (\ref{generalized HJB equation-finite dimensional-4}) with the terminal condition reduces to
\begin{equation}\label{generalized HJB equation-finite dimensional-final}
\left\{
\begin{aligned}
&-\frac{\partial V}{\partial s}(s,x,x_1)-\Big[b_1(s,x,x_1,u^*)+e^{\lambda\delta}(x-\lambda x_1)b_2(s,x,x_1,u^*)\Big]\frac{\partial V}{\partial x}(s,x,x_1)\\
&\ -\frac{1}{2}\sigma^2(s,x,x_1,u^*)\frac{\partial^2V}{\partial x^2}(s,x,x_1)\\
&\ +f_1\Big(s,x,x_1,-V(s,x,x_1),-\sigma(s,x,x_1,u^*)\frac{\partial V}{\partial x}(s,x,x_1),u^*\Big)=0,\\
&V(T,x,x_1)=-\phi(x,x_1),
\end{aligned}
\right.
\end{equation}
which is independent of $x_2$. Note that now (\ref{differentiating HJB equation with respect to x2}) takes the form
\begin{equation}\label{differentiating HJB equation with respect to x2-2}
\begin{aligned}
&\ \frac{\partial V}{\partial x_1}(s,x,x_1)
 +e^{\lambda\delta}\bigg[f_2\Big(s,x,x_1,-V(s,x,x_1),-\sigma(s,x,x_1,u^*)\frac{\partial V}{\partial x}(s,x,x_1),u^*\Big)\\
&\qquad\qquad\qquad\qquad\ -b_2(s,x,x_1,u^*)\frac{\partial V}{\partial x}(s,x,x_1)\bigg]=0.
\end{aligned}
\end{equation}
We introduce new variables $\tilde{x}$ and $\tilde{x}_1$, such that
\begin{equation*}
\left\{
\begin{aligned}
\frac{\partial}{\partial\tilde{x}_1}&:=\frac{\partial}{\partial x_1}
 +e^{\lambda\delta}\bigg[f_2\Big(s,x,x_1,-V(s,x,x_1),-\sigma(s,x,x_1,u^*)\frac{\partial V}{\partial x}(s,x,x_1),u^*\Big)\\
&\qquad-b_2(s,x,x_1,u^*)\frac{\partial}{\partial x}\bigg],\qquad\frac{\partial}{\partial\tilde{x}}:=\frac{\partial}{\partial x}.
\end{aligned}
\right.
\end{equation*}
Then (\ref{differentiating HJB equation with respect to x2-2}) states that
\begin{equation*}
\frac{\partial V}{\partial\tilde{x}_1}(s,\tilde{x},\tilde{x}_1)=0,\mbox{ for all }(s,\tilde{x},\tilde{x}_1).
\end{equation*}
In order to be compatible with this, the coefficients of (\ref{generalized HJB equation-finite dimensional-final}) and the functions $f_1,\phi$ must also be constants in $\tilde{x}_1$, that is
\begin{equation}\label{system of first order PDEs-initial form}
\left\{
\begin{aligned}
&\frac{\partial\hat{b}}{\partial x_1}(s,x,x_1,u^*)
 +e^{\lambda\delta}\bigg[f_2\Big(s,x,x_1,-V(s,x,x_1),-\sigma(s,x,x_1,u^*)\frac{\partial V}{\partial x}(s,x,x_1),u^*\Big)\\
&\quad-b_2(s,x,x_1,u^*)\frac{\partial\hat{b}}{\partial x}(s,x,x_1,u^*)\bigg]=0,\\
&\frac{\partial\sigma}{\partial x_1}(s,x,x_1,u^*)
 +e^{\lambda\delta}\bigg[f_2\Big(s,x,x_1,-V(s,x,x_1),-\sigma(s,x,x_1,u^*)\frac{\partial V}{\partial x}(s,x,x_1),u^*\Big)\\
&\quad-b_2(s,x,x_1,u^*)\frac{\partial\sigma}{\partial x}(s,x,x_1,u^*)\bigg]=0,\\
&\frac{\partial f_1}{\partial x_1}\Big(s,x,x_1,-V(s,x,x_1),-\sigma(s,x,x_1,u^*)\frac{\partial V}{\partial x}(s,x,x_1),u^*\Big)\\
&+e^{\lambda\delta}\bigg[f_2\Big(s,x,x_1,-V(s,x,x_1),-\sigma(s,x,x_1,u^*)\frac{\partial V}{\partial x}(s,x,x_1),u^*\Big)\\
&\qquad\quad-b_2(s,x,x_1,u^*)\frac{\partial f_1}{\partial x}\Big(s,x,x_1,-V(s,x,x_1),-\sigma(s,x,x_1,u^*)\\
&\qquad\qquad\qquad\qquad\qquad\qquad\times\frac{\partial V}{\partial x}(s,x,x_1),u^*\Big)\bigg]=0,\\
&\frac{\partial\phi}{\partial x_1}(x,x_1)
 +e^{\lambda\delta}\bigg[f_2\Big(s,x,x_1,-V(s,x,x_1),-\sigma(s,x,x_1,u^*)\frac{\partial V}{\partial x}(s,x,x_1),u^*\Big)\\
&\qquad\qquad\qquad\quad-b_2(s,x,x_1,u^*)\frac{\partial\phi}{\partial x}(x,x_1)\bigg]=0,\quad\forall (s,x,x_1),
\end{aligned}
\right.
\end{equation}
where
\begin{equation*}
\hat{b}(s,x,x_1,u)):=b_1(s,x,x_1,u)+e^{\lambda\delta}(x-\lambda x_1)b_2(s,x,x_1,u).
\end{equation*}
Using the initial variables
\begin{equation*}
y\equiv y(s,x,x_1)=-V(s,x,x_1),\ z\equiv z(s,x,x_1)=-\sigma(s,x,x_1,u^*)\frac{\partial V}{\partial x}(s,x,x_1)
\end{equation*}
in (\ref{system of first order PDEs-initial form}), we end with (\ref{system of first order PDEs}). The proof is complete.
\end{proof}

\begin{remark}
If the condition (\ref{assumption of b,sigma,f}) holds, then for given $(s,\varphi)\in[0,T)\times C([-\delta,0];\mathbb{R})$ and control $u(\cdot)\in\mathcal{U}[s,T]$, controlled FBSDDE (\ref{controlled FBSDDE}) takes the following special form
\begin{equation}\label{controlled FBSDDE-finite dimensional-1}
\left\{
\begin{aligned}
 dX^{s,\varphi;u}(t)&=\Big[b_1(t,X^{s,\varphi;u}(t),X_1^{s,\varphi;u}(t),u(t))\\
                    &\quad+b_2(t,X^{s,\varphi;u}(t),X_1^{s,\varphi;u}(t),u(t))X_2^{s,\varphi;u}(t)\Big]dt\\
                    &\quad+\sigma(t,X^{s,\varphi;u}(t),X_1^{s,\varphi;u}(t),u(t))dW(t),\\
-dY^{s,\varphi;u}(t)&=\Big[f_1\big(t,X^{t,\varphi;u}(t),X_1^{t,\varphi;u}(t),Y^{s,\varphi;u}(t),Z^{s,\varphi;u}(t),u(t)\big)\\
                    &\quad\ +f_2\big(t,X^{t,\varphi;u}(t),X_1^{t,\varphi;u}(t),Y^{s,\varphi;u}(t),Z^{s,\varphi;u}(t),u(t)\big)X_2^{s,\varphi;u}(t)\Big]dt\\
                      &\quad-Z^{t,\varphi;u}(t)dW(t),\quad t\in[s,T],\\
  X^{s,\varphi;u}(t)&=\varphi(t-s),\quad t\in[s-\delta,s],\\
  Y^{s,\varphi;u}(T)&=\phi(X^{s,\varphi;u}(T),X_1^{s,\varphi;u}(T)).
\end{aligned}
\right.
\end{equation}
That is to say, condition (\ref{assumption of b,sigma,f}) together with the PDEs system (\ref{system of first order PDEs}) guarantees that the reduction of PDE (\ref{BSEE}) from an infinite dimensional one to its finite dimensional counterpart (\ref{generalized HJB equation}). Though the results obtained in Theorems 2.6 and 2.7 corresponding to (\ref{controlled FBSDDE-finite dimensional-1}) are less general than (\ref{controlled FBSDDE}), they never the less cover many interesting applications. In Section 4, we will present one financial example that satisfy the conditions (\ref{assumption of b,sigma,f}), (\ref{system of first order PDEs}) for its dynamics of the state being the form of (\ref{controlled FBSDDE-finite dimensional-1}). Some discussions are also given to indicate why it is difficult to find more general examples.
\end{remark}

\section{Relationship with Maximum Principle}

For stochastic optimal control problems with time delay and those of FBSDEs (recursive utility, without time delay), the relationships between dynamic programming principle and maximum principle are shown in \cite{Shi11} and \cite{ShiYu13}, respectively. In this section, a similar relationship is given between the value function $V$, the generalized Hamiltonian function $G$, and the adjoint processes $\vec{p},\vec{q}$ (see Theorem 3.1), under the assumption that the value function is smooth enough and depends on the initial path of the state in a simple way as in Theorem 2.6. The main result is shown in Theorem 3.2, which could cover many interesting applications.

For this target, we first solve {\bf Problem (FBSOCPD)} by the Pontryagin's maximum principle approach. In this part, let the initial time $s=0$ and write $X=X^u=X^{0,\varphi;u}$, etc. Moreover, we need the following additional assumptions.

\noindent(H7)\quad The functions $b(t,x,x_1,x_2,u),\sigma(t,x,x_1,x_2,u)$
are continuously differentiable in $(x,x_1,x_2,u)$, such that
\begin{equation*}
\mathbb{E}\int_0^T\Big(\Big|\frac{\partial b}{\partial \lambda}(t,X(t),X_1(t),X_2(t),u(t))\Big|^2
+\Big|\frac{\partial\sigma}{\partial\lambda}(t,X(t),X_1(t),X_2(t),u(t))\Big|^2\Big)dt<\infty,
\end{equation*}
for $\lambda=x,x_1,x_2,u$.

\noindent(H8)\quad The function $f(t,x,x_1,x_2,y,z,u)$ is continuously differentiable in $(x,x_1,x_2,\\y,z,u)$ and $\phi(x,x_1)$ is continuously differentiable in $(x,x_1)$, such that
\begin{equation*}
\begin{aligned}
&\mathbb{E}\bigg\{\int_0^T\Big|\frac{\partial f}{\partial\lambda}(t,X(t),X_1(t),X_2(t),Y(t),Z(t),u(t))\Big|^2dt\\
&\quad+\Big|\frac{\partial\phi}{\partial x}(X(T),X_1(T))\Big|^2+\Big|\frac{\partial\phi}{\partial x_1}(X(T),X_1(T))\Big|^2\bigg\}<\infty,
\end{aligned}
\end{equation*}
for $\lambda=x,x_1,x_2,y,z,u$.

\vspace{1mm}

We introduce the Hamiltonian function
$H:[0,T]\times\mathbb{R}^5\times\mathbb{U}\times\mathbb{R}^3\times\mathbb{R}\times\mathbb{R}^2\rightarrow\mathbb{R}$
as
\begin{equation}\label{Hamiltonian function}
\begin{aligned}
&H(t,x,x_1,x_2,y,z,u,\vec{p},q,\vec{k}):=p_1b(t,x,x_1,x_2,u)\\
&\quad+p_2\big(x-\lambda
x_1-e^{-\lambda\delta}x_2\big)+k_1\sigma(t,x,x_1,x_2,u)-qf(t,x,x_1,x_2,y,z,u),
\end{aligned}
\end{equation}
where $\vec{p}=(p_1,p_2,p_3)^\top,\vec{k}=(k_1,k_2)^\top$. For each $u^*(\cdot)\in\mathcal{U}[0,T]$ and the
corresponding solutions $(X^*(\cdot),Y^*(\cdot),Z^*(\cdot))$ to
(\ref{controlled FBSDDE}) with (\ref{definition of delay terms}), we also introduce the following adjoint equation
\begin{equation}\label{adjoint equation}
\left\{
\begin{aligned}
-dp_1(t)&=\frac{\partial H^*}{\partial x}(t)dt-k_1(t)dW(t),\\
-dp_2(t)&=\frac{\partial H^*}{\partial x_1}(t)dt-k_2(t)dW(t),\quad-dp_3(t)=\frac{\partial H^*}{\partial x_2}(t)dt,\\
   dq(t)&=-\frac{\partial H^*}{\partial y}(t)dt-\frac{\partial H^*}{\partial z}(t)dW(t),\quad t\in[0,T],\\
    q(0)&=1,\ p_1(T)=-\frac{\partial\phi}{\partial x}(X^*(T),X_1^*(T))q(T),\\\
  p_2(T)&=-\frac{\partial\phi}{\partial x_1}(X^*(T),X_1^*(T))q(T),\ p_3(T)=0,
\end{aligned}
\right.
\end{equation}
where for notation simplicity, we use
\begin{equation*}
   H^*(t)=H\big(t,X^*(t),X_1^*(t),X_2^*(t),Y^*(t),Z^*(t),u^*(t),\vec{p}(t),q(t),\vec{k}(t)\big),
\end{equation*}
and all its partial derivatives. From the definition (\ref{Hamiltonian function}), under assumptions (H7), (H8), controlled FBSDDE (\ref{adjoint equation}) admits a unique $\mathcal{F}_t^0$-adapted solution $(\vec{p}(\cdot)\equiv(p_1(\cdot),\\p_2(\cdot),p_3(\cdot))^\top,q(\cdot),\vec{k}(\cdot)\equiv(k_1(\cdot),k_2(\cdot))^\top)$. (\ref{adjoint equation}) is called a stochastic Hamiltonian system. Note that the adjoint equation (\ref{adjoint equation}) introduced here is similar to those in \cite{OS00} and Section 3.3 in \cite{AHOP13}. For practicable importance, it is easier to be solved explicitly than the time advanced ones in \cite{ChenWu10} and \cite{OSZ11}.

The following is the sufficient maximum principle as a verification result. Though its proof is similar to those in \cite{OS00}, we give the detail for completeness.

\begin{theorem}
{\bf (Sufficient Maximum Principle)} Let $u^*(\cdot)\in\mathcal{U}[0,T]$, $(X^*(\cdot),\\Y^*(\cdot),Z^*(\cdot))$ be the corresponding solution to
(\ref{controlled FBSDDE}) with (\ref{definition of delay terms}), and
$\vec{p}(\cdot)=(p_1(\cdot),p_2(\cdot),\\p_3(\cdot))^\top,q(\cdot),\vec{k}(\cdot)=(k_1(\cdot),k_2(\cdot))^\top$ be the corresponding solutions to
(\ref{adjoint equation}). Suppose that
\begin{equation}\label{concvexity assumption}
\begin{aligned}
(x,x_1,x_2,y,z,u)\rightarrow H(t,x,x_1,x_2,y,z,u,\vec{p}(t),q(t),\vec{k}(t))\mbox{ is convex},\ \forall t\in[0,T],
\end{aligned}
\end{equation}
\begin{equation}\label{phi's special case}
\phi(x,x_1)=Mx+Nx_1,\ \mbox{ for }M,N\in\mathbb{R},
\end{equation}
\begin{equation}\label{p3=0}
p_3(t)=0,\mbox{ for all }t\in[0,T],
\end{equation}
and
\begin{equation}\label{maximum condition}
\begin{aligned}
&\frac{\partial H^*}{\partial u}(t)\big(u^*(t)-u\big)\leq0,\quad\forall u\in\mathbb{U}, a.s.,\mbox{ for all }t\in[0,T].
\end{aligned}
\end{equation}
Then $u^*(\cdot)$ is an optimal control for {\bf Problem (FBSOCPD)}.
\end{theorem}

\begin{proof}
For any $u(\cdot)\in\mathcal{U}[0,T]$, with the corresponding solution $(X(\cdot),Y(\cdot),Z(\cdot))$ to
(\ref{controlled FBSDDE}) with (\ref{definition of delay terms}), applying delayed It\^{o}'s formula (\ref{Ito's formula}) to $q(t)(Y^*(t)-Y(t))+p_1(t)(X^*(t)-X(t))+p_2(t)\int_{-\delta}^0e^{\lambda\tau}(X^*(t+\tau)-X(t+\tau)d\tau+p_3(t)(X^*(t-\delta)-X(t-\delta))$, noting that
\begin{equation}\label{applying Ito's formula to X-1}
d\int_{-\delta}^0e^{\lambda\tau}X(t+\tau)d\tau=\Big[X(t)-e^{-\lambda\delta}X(t-\delta)-\lambda\int_{-\delta}^0e^{\lambda\tau}X(t+\tau)d\tau\Big]dt,
\end{equation}
we have
\begin{equation*}
\begin{aligned}
 &\mathbb{E}\big[q(T)(Y^*(T)-Y(T))\big]-\big[Y^*(0)-Y(0)\big]\\
 &\ -\mathbb{E}\big[Mq(T)(X^*(T)-X(T))\big]-\mathbb{E}\big[Nq(T)(X_1^*(T)-X_1(T))\big]\\
=&\ -\big[Y^*(0)-Y(0)\big]=J(0,\varphi;u^*(\cdot))-J(0,\varphi;u(\cdot))\\
=&\ \mathbb{E}\int_0^T\bigg\{p_1(t)\bigg(\frac{\partial b^*}{\partial u}(t)-\frac{\partial b}{\partial u}(t,X(t),X_1(t),X_2(t),u(t))\bigg)\\
 &\qquad\quad+k_1(t)\bigg(\frac{\partial\sigma^*}{\partial u}(t)-\frac{\partial\sigma}{\partial u}(t,X(t),X_1(t),X_2(t),u(t))\bigg)\\
 &\qquad\quad-q(t)\bigg(\frac{\partial f^*}{\partial u}(t)-\frac{\partial f}{\partial u}(t,X(t),X_1(t),Y(t),Z(t),u(t))\bigg)\\
 &\qquad\quad-p_1(t)\frac{\partial b^*}{\partial x}(t)\big(X^*(t)-X(t)\big)-k_1(t)\frac{\partial\sigma^*}{\partial x}(t)\big(X^*(t)-X(t)\big)\\
 &\qquad\quad+q(t)\frac{\partial f^*}{\partial x}(t)\big(X^*(t)-X(t)\big)-p_1(t)\frac{\partial b^*}{\partial x_1}(t)\big(X_1^*(t)-X_1(t)\big)\\
 &\qquad\quad-k_1(t)\frac{\partial\sigma^*}{\partial x_1}(t)\big(X_1^*(t)-X_1(t)\big)+q(t)\frac{\partial f^*}{\partial x_1}(t)\big(X_1^*(t)-X_1(t)\big)\\
 &\qquad\quad-p_1(t)\frac{\partial b^*}{\partial x_2}(t)\big(X_2^*(t)-X_2(t)\big)-k_1(t)\frac{\partial\sigma^*}{\partial x_2}(t)\big(X_2^*(t)-X_2(t)\big)\\
 &\qquad\quad+q(t)\frac{\partial f^*}{\partial y}(t)\big(Y^*(t)-Y(t)\big)+q(t)\frac{\partial f^*}{\partial z}(t)\big(Z^*(t)-Z(t)\big)\bigg\}dt\\
=&\ \mathbb{E}\int_0^T\bigg\{H^*(t)-H\big(t,X(t),X_1(t),X_2(t),Y(t),Z(t),u(t),\vec{p}(t),q(t),\vec{k}(t)\big)\\
 &\qquad\quad-\frac{\partial H^*}{\partial x}(t)\big(X^*(t)-X(t)\big)-\frac{\partial H^*}{\partial x_1}(t)\big(X_1^*(t)-X_1(t)\big)\\
 &\qquad\quad-\frac{\partial H^*}{\partial x_2}(t)\big(X_2^*(t)-X_2(t)\big)-\frac{\partial H^*}{\partial y}(t)\big(Y^*(t)-Y(t)\big)\\
 &\qquad\quad-\frac{\partial H^*}{\partial z}(t)\big(Z^*(t)-Z(t)\big)\bigg\}dt.
\end{aligned}
\end{equation*}
In the above, we have used
\begin{equation*}
\begin{aligned}
  &\frac{\partial\beta^*}{\partial \rho}(t)=\frac{\partial\beta}{\partial \rho}(t,X^*(t),X_1^*(t),X_2^*(t),u^*(t)),\\
\end{aligned}
\end{equation*}
for $\beta=b,\sigma,f$ and $\rho=x,x_1,x_2,y,z,u$. By (\ref{concvexity assumption}) we then have
\begin{equation*}
J(0,\varphi;u^*(\cdot))-J(0,\varphi;u(\cdot))\leq\mathbb{E}\int_0^T\frac{\partial H^*}{\partial u}(t)\big(u^*(t)-u(t)\big)dt\leq0.
\end{equation*}
Thus $u^*(\cdot)$ is an optimal control for {\bf Problem (FBSOCPD)}. The proof is complete.
\end{proof}

\begin{proof}
Note that this sufficient maximum principle is proved for controlled system (\ref{controlled FBSDDE}) other than its special form (\ref{controlled FBSDDE-finite dimensional}), and in the special case that $\phi$ is linear with respect to $x,x_1$ (see (\ref{phi's special case})). The general case to eliminate this linear restriction is open, even for problem of FBSDEs without time delay. See \cite{OS09}, \cite{ShiWu10} for details.
\end{proof}

The following is the main result in this section.

\begin{theorem}
Let (H1)$\sim$(H8)
hold and $(s,x,x_1)\in[0,T)\times\mathbb{R}^2$ be fixed. Suppose that
$u^*(\cdot)$ is an optimal control for {\bf Problem (SROCPD)} and $(X^*(\cdot),Y^*(\cdot),Z^*(\cdot))$ is
the corresponding optimal state which is the solution to (\ref{controlled FBSDDE}) with (\ref{definition of delay terms}). Let
$\vec{p}(\cdot)=(p_1(\cdot),p_2(\cdot),p_3(\cdot))^\top,q(t),\vec{k}(\cdot)=(k_1(\cdot),k_2(\cdot))^\top$ be the solution to adjoint
equation (\ref{adjoint equation}). Assume that the value function $V$ depends on $(s,x,x_1)$ only and $V(s,x,x_1)\in
C^{1,2,1}([0,T]\times\mathbb{R}^2)$, then
\begin{equation}\label{relation with respect time variable}
\begin{aligned}
&\ \frac{\partial V}{\partial t}(t,X^*(t),X_1^*(t))\\
=&\ G\Big(t,X^*(t),X_1^*(t),X_2^*(t),u^*(t),-V(t,X^*(t),X_1^*(t)),-\frac{\partial V}{\partial x}(t,X^*(t),X_1^*(t)),\\
&\quad\ -\frac{\partial^2V}{\partial x^2}(t,X^*(t),X_1^*(t)),-\frac{\partial V}{\partial x_1}(t,X^*(t),X_1^*(t))\Big)\\
=&\ \max\limits_{u\in\mathbf{U}}G\Big(t,X^*(t),X_1^*(t),X_2^*(t),u,-V(t,X^*(t),X_1^*(t)),-\frac{\partial V}{\partial x}(t,X^*(t),X_1^*(t)),\\
&\qquad\quad -\frac{\partial^2V}{\partial x^2}(t,X^*(t),X_1^*(t)),-\frac{\partial V}{\partial x_1}(t,X^*(t),X_1^*(t))\Big),\ a.e.t\in[s,T],a.s.
\end{aligned}
\end{equation}
Further, if $V(s,x,x_1)\in C^{1,3,2}([0,T]\times\mathbb{R}^2)$ and
$\frac{\partial^2V}{\partial t\partial x},\frac{\partial^2V}{\partial t\partial x_1},\frac{\partial^2V}{\partial x\partial x_1},\frac{\partial^3V}{\partial x^2\partial x_1}$ are also continuous, then
\begin{equation}\label{relation with respect state variable}
\left\{
\begin{aligned}
p_1(t)&=\frac{\partial V}{\partial x}(t,X^*(t),X_1^*(t))q(t),\quad\forall t\in[s,T],a.s.,\\
k_1(t)&=\bigg[\frac{\partial^2V}{\partial x^2}(t,X^*(t),X_1^*(t))\sigma^*(t)
       +\frac{\partial V}{\partial x}(t,X^*(t),X_1^*(t))\frac{\partial f^*}{\partial z}(t)\Big)\bigg]q(t),\\
      &\hspace{35mm}a.e.t\in[s,T],a.s.,\\
p_2(t)&=\frac{\partial V}{\partial x_1}(t,X^*(t),X_1^*(t))q(t),\quad\forall t\in[s,T],a.s.,\\
k_2(t)&=\bigg[\frac{\partial^2V}{\partial x\partial x_1}(t,X^*(t),X_1^*(t))\sigma^*(t)
       +\frac{\partial V}{\partial x_1}(t,X^*(t),X_1^*(t))\frac{\partial f^*}{\partial z}(t)\Big)\bigg]q(t),\\
      &\hspace{35mm}a.e.t\in[s,T],a.s.,\\
  q(t)&=\exp\bigg\{\int_s^t\frac{\partial f^*}{\partial y}(r)dW(r)-\frac{1}{2}\int_s^t\Big|\frac{\partial f^*}{\partial z}(r)\Big|^2dr\bigg\},\quad\forall t\in[s,T],a.s.,
\end{aligned}
\right.
\end{equation}
where for notational simplicity, we have used for all $t\in[s,T]$,
\begin{equation*}
\left\{
\begin{aligned}
&b^*(t)\equiv b(t,X^*(t),X_1^*(t),X_2^*(t),u^*(t)),\ \sigma^*(t)\equiv\sigma(t,X^*(t),X_1^*(t),X_2^*(t),u^*(t)),\\
&\frac{\partial f^*}{\partial y}(t)\equiv\frac{\partial f}{\partial z}\Big(t,X^*(t),X_1^*(t),X_2^*(t),-V(t,X^*(t),X_1^*(t),\\
&\hspace{2.5cm}\frac{\partial V}{\partial x}(t,X^*(t),X_1^*(t))\sigma^*(t),u^*(t)\Big),\\
&\frac{\partial f^*}{\partial z}(t)\equiv\frac{\partial f}{\partial z}\Big(t,X^*(t),X_1^*(t),X_2^*(t),-V(t,X^*(t),X_1^*(t),\\
&\hspace{2.5cm}\frac{\partial V}{\partial x}(t,X^*(t),X_1^*(t))\sigma^*(t),u^*(t)\Big).
\end{aligned}
\right.
\end{equation*}
\end{theorem}
For simplicity, in this section we write $X=X^{s,x,x_1;u}$ and $X^*=X^{s,x,x_1;u^*}$, etc.
\begin{proof}
The proof method is inspired from Chapter 4, \cite{YZ99} and \cite{ShiYu13}. First, by (\ref{generalized dynamic programming principle by Chen and Wu}), similar to \cite{ChenWu12}, we can get that for any $t\in[s,T]$,
\begin{equation*}
\begin{aligned}
&V(t,X^*(t),X_1^*(t))=-Y^*(t)\\&=-\mathbb{E}^{s,x,x_1;u^*}\bigg[\int_t^Tf\big(r,X^*(r),X_1^*(r),X_2^*(r),Y^*(r),Z^*(r),u^*(r)\big)dr\\
&\hspace{2.2cm}+\phi(X^*(T),X_1^*(T))\bigg].
\end{aligned}
\end{equation*}
Define a square-integrable $\mathcal{F}_t^s$-martingale (recall that $s\in[0,T)$ is fixed)
\begin{equation*}
\begin{aligned}
m(t)&:=-\mathbb{E}\Big[\int_s^Tf\big(r,X^*(r),X_1^*(r),X_2^*(r),Y^*(r),Z^*(r),u^*(r)\big)dr\\
&\hspace{1.5cm}+\phi(X^*(T),X_1^*(T))\Big|\mathcal{F}_t^s\Big],
\end{aligned}
\end{equation*}
for any $t\in[s,T]$. Thus, by the martingale representation theorem, there exists a
unique $M(\cdot)\in L_\mathcal{F}^2([s,T];\mathbb{R})$ such that
\begin{equation*}
m(t)=m(s)+\int_s^tM(r)dW(r)=V(s,x,x_1)+\int_s^tM(r)dW(r),\quad t\in[s,T].
\end{equation*}
So
\begin{equation*}
\begin{aligned}
 V(t,X^*(t),X_1^*(t))=&\ -\int_t^Tf\big(r,X^*(r),X_1^*(r),X_2^*(r),Y^*(r),Z^*(r),u^*(r)\big)dr\\
                    &\ -\int_s^TM(r)dW(r)+V(T,X^*(T),X_1^*(T)).
\end{aligned}
\end{equation*}
On the other hand, applying delayed It\^{o}'s formula (\ref{Ito's formula}) to $V(t,X^*(t),X_1^*(t))$, we obtain
\begin{equation*}\label{applying Ito's formula to V(t)}
\begin{aligned}
 &dV(t,X^*(t),X_1^*(t))=\bigg\{\frac{\partial V}{\partial t}(t,X^*(t),X_1^*(t))+b^*(t)\frac{\partial V}{\partial x}(t,X^*(t),X_1^*(t))\\
 &\quad+\frac{1}{2}\sigma^{*2}(t)\frac{\partial^2V}{\partial x^2}(t,X^*(t),X_1^*(t))\\
 &\quad+\big[X^*(t)-\lambda X_1^*(t)-e^{-\lambda\delta}X_2^*(t)\big]\frac{\partial V}{\partial x_1}(t,X^*(t),X_1^*(t))\bigg\}dt\\
 &\quad+\sigma^*(t)\frac{\partial V}{\partial x}(t,X^*(t),X_1^*(t))dW(t).
\end{aligned}
\end{equation*}
Comparing the above two equations, we conclude that
\begin{equation}\label{relation}
\left\{
\begin{aligned}
 &\frac{\partial V}{\partial t}(t,X^*(t),X_1^*(t))+b^*(t)\frac{\partial V}{\partial x}(t,X^*(t),X_1^*(t))+\frac{1}{2}\sigma^{*2}(t)\frac{\partial^2V}{\partial x^2}(t,X^*(t),X_1^*(t))\\
 &\quad+\big(X^*(t)-\lambda X_1^*(t)-e^{-\lambda\delta}X_2^*(t)\big)\frac{\partial V}{\partial x_1}(t,X^*(t),X_1^*(t))\\
 &=f\big(t,X^*(t),X_1^*(t),X_2^*(t),Y^*(t),Z^*(t),u^*(t)\big),\quad\forall t\in[s,T],a.s.,\\
 &\sigma^*(t)\frac{\partial V}{\partial x}(t,X^*(t),X_1^*(t))=M(t),\quad a.e.\ t\in[s,T],a.s.
\end{aligned}
\right.
\end{equation}
However, by the uniqueness of solution to BSDDE (\ref{controlled BSDDE}), we have
\begin{equation*}
\left\{
\begin{aligned}
&Y^*(t)=-V(t,X^*(t),X_1^*(t)),\\
&Z^*(t)=-\frac{\partial V}{\partial x}(t,X^*(t),X_1^*(t))\sigma^*(t),\quad a.e.\ t\in[s,T],a.s.
\end{aligned}
\right.
\end{equation*}

Since $V(s,x,x_1)\in C^{1,2,1}([0,T]\times\mathbb{R}^2)$, it satisfies the generalized HJB equation (\ref{generalized HJB equation}) by Theorem 2.6, which implies (\ref{relation with respect time variable}). Also, by (\ref{generalized HJB equation}), we have
\begin{equation}\label{Yong and Zhou's equality}
\begin{aligned}
  0=&\ -\frac{\partial V}{\partial t}(t,X^*(t),X_1^*(t))+G\Big(t,X^*(t),X_1^*(t),X_2^*(t),u^*(t),-V(t,X^*(t),X_1^*(t)),\\
    &\qquad-\frac{\partial V}{\partial x}(t,X^*(t),X_1^*(t)),-\frac{\partial^2V}{\partial x^2}(t,X^*(t),X_1^*(t)),
     -\frac{\partial V}{\partial x_1}(t,X^*(t),X_1^*(t))\Big)\\
\geq&\ -\frac{\partial V}{\partial t}(t,x,x_1)+G\Big(t,x,x_1,x_2,u^*(t),-V(t,x,x_1),-\frac{\partial V}{\partial x}(t,x,x_1),\\
    &\qquad-\frac{\partial^2V}{\partial x^2}(t,x,x_1),-\frac{\partial V}{\partial x_1}(t,x,x_1)\Big),\quad \forall (x,x_1,x_2)\in\mathbb{R}^3.
\end{aligned}
\end{equation}
Consequently, if $V\in C^{1,3,2}([0,T]\times\mathbb{R}^2)$ and $\frac{\partial^2V}{\partial t\partial x},\frac{\partial^2V}{\partial x\partial x_1}$ are continuous, then
\begin{equation*}
\begin{aligned}
 &\frac{\partial}{\partial x}\bigg\{-\frac{\partial V}{\partial t}(t,x,x_1)+G\Big(t,x,x_1,x_2,u^*(t),-V(t,x,x_1),\\
 &\qquad-\frac{\partial V}{\partial x}(t,x,x_1),-\frac{\partial^2V}{\partial x^2}(t,x,x_1),
  -\frac{\partial V}{\partial x_1}(t,x,x_1)\Big)\bigg\}\bigg|_{x=X^*(t)}=0,\quad t\in[s,T].
\end{aligned}
\end{equation*}
This is equivalent to (recall the definition of $G$ in (\ref{generalized Hamiltonian function}))
\begin{equation*}
\begin{aligned}
 &\quad-\frac{\partial^2V}{\partial t\partial x}(t,X^*(t),X_1^*(t))-\frac{\partial^2V}{\partial x^2}(t,X^*(t),X_1^*(t))b^*(t)\\
 &\quad-\frac{\partial V}{\partial x}(t,X^*(t),X_1^*(t))\frac{\partial b^*}{\partial x}(t)-\frac{\partial V}{\partial x_1}(t,X^*(t),X_1^*(t))\\
 &\quad-\frac{1}{2}\sigma^{*2}(t)\frac{\partial^3V}{\partial x^3}(t,X^*(t),X_1^*(t))
  -\frac{\partial\sigma^*}{\partial x}(t)\frac{\partial^2V}{\partial x^2}(t,X^*(t),X_1^*(t))\sigma^*(t)\\
 &\quad-\big(X^*(t)-\lambda X_1^*(t)-e^{-\lambda\delta}X_2^*(t)\big)\frac{\partial^2V}{\partial x\partial x_1}(t,X^*(t),X_1^*(t))\\
 &\quad+\frac{\partial f^*}{\partial x}(t)-\frac{\partial f^*}{\partial y}(t)\frac{\partial V}{\partial x}(t,X^*(t),X_1^*(t))\\
 &\quad-\frac{\partial f^*}{\partial z}(t)\bigg[\frac{\partial^2V}{\partial x^2}(t,X^*(t),X_1^*(t))\sigma^*(t)
  +\frac{\partial V}{\partial x}(t,X^*(t),X_1^*(t))\frac{\partial\sigma^*}{\partial x}(t)\bigg]=0,
\end{aligned}
\end{equation*}
$\forall t\in[s,T]$. On the other hand, applying (\ref{Ito's formula}) to $\frac{\partial V}{\partial x}(t,X^*(t),X_1^*(t))$, we get
\begin{equation*}
\begin{aligned}
&\qquad d\frac{\partial V}{\partial x}(t,X^*(t),X_1^*(t))\\
&=\bigg\{\frac{\partial^2V}{\partial t\partial x}(t,X^*(t),X_1^*(t))+\frac{\partial^2V}{\partial x^2}(t,X^*(t),X_1^*(t))b^*(t)\\
&\qquad+\frac{1}{2}\sigma^{*2}(t)\frac{\partial^3V}{\partial x^3}(t,X^*(t),X_1^*(t))\\
&\qquad+\big[X^*(t)-\lambda X_1^*(t)-e^{-\lambda\delta}X_2^*(t)\big]\frac{\partial^2V}{\partial x\partial x_1}(t,X^*(t),X_1^*(t))\bigg\}dt\\
&\quad+\frac{\partial^2V}{\partial x^2}(t,X^*(t),X_1^*(t))\sigma^*(t)dW(t)\\
&=\bigg\{-\frac{\partial V}{\partial x}(t,X^*(t),X_1^*(t))\frac{\partial b^*}{\partial x}(t)
 -\frac{\partial\sigma^*}{\partial x}(t)\frac{\partial^2V}{\partial x^2}(t,X^*(t),X_1^*(t))\sigma^*(t)\\
&\qquad-\frac{\partial V}{\partial x_1}(t,X^*(t),X_1^*(t))+\frac{\partial f^*}{\partial x}(t)
 -\frac{\partial f^*}{\partial y}(t)\frac{\partial V}{\partial x}(t,X^*(t),X_1^*(t))\\
&\quad-\frac{\partial f^*}{\partial z}(t)\bigg[\frac{\partial^2V}{\partial x^2}(t,X^*(t),X_1^*(t))\sigma^*(t)
  +\frac{\partial V}{\partial x}(t,X^*(t),X_1^*(t))\frac{\partial\sigma^*}{\partial x}(t)\bigg]\bigg\}dt\\
&\quad+\frac{\partial^2V}{\partial x^2}(t,X^*(t),X_1^*(t))\sigma^*(t)dW(t).
\end{aligned}
\end{equation*}
Applying again (\ref{Ito's formula}) to $\frac{\partial V}{\partial x}(t,X^*(t),X_1^*(t))q(t)$, noting that
\begin{equation*}
\frac{\partial V}{\partial x}(T,X^*(T),X_1^*(T))=-\frac{\partial\phi}{\partial x}(X^*(T),X_1^*(T)),
\end{equation*}
we have
\begin{equation*}
\begin{aligned}
&d\bigg\{\frac{\partial V}{\partial x}(t,X^*(t),X_1^*(t))q(t)\bigg\}
=\bigg\{-\frac{\partial b^*}{\partial x}(t)\frac{\partial V}{\partial x}(t,X^*(t),X_1^*(t))q(t)\\
&\qquad-\frac{\partial\sigma^*}{\partial x}(t)\bigg[\frac{\partial^2V}{\partial x^2}(t,X^*(t),X_1^*(t))\sigma^*(t)
 +\frac{\partial f^*}{\partial z}(t)\frac{\partial V}{\partial x}(t,X^*(t),X_1^*(t))\bigg]q(t)\\
&\qquad-\frac{\partial V}{\partial x_1}(t,X^*(t),X_1^*(t))q(t)+\frac{\partial f^*}{\partial x}(t)q(t)\bigg\}dt\\
&\quad+\bigg[\frac{\partial^2V}{\partial x^2}(t,X^*(t),X_1^*(t))\sigma^*(t)
 +\frac{\partial f^*}{\partial z}(t)\frac{\partial V}{\partial x}(t,X^*(t),X_1^*(t))\bigg]q(t)dW(t).
\end{aligned}
\end{equation*}
Hence, by the uniqueness of the solution to the $p_1(t)$ part of (\ref{adjoint equation}), we have
\begin{equation}\label{relation of p1}
\left\{
\begin{aligned}
p_1(t)&=\frac{\partial V}{\partial x}(t,X^*(t),X_1^*(t))q(t),\quad\forall t\in[s,T],a.s.,\\
k_1(t)&=\bigg[\frac{\partial^2V}{\partial x^2}(t,X^*(t),X_1^*(t))\sigma^*(t)
       +\frac{\partial V}{\partial x}(t,X^*(t),X_1^*(t))\frac{\partial f^*}{\partial z}(t)\Big)\bigg]q(t),\\
      &\hspace{55mm}a.e.t\in[s,T],a.s.
\end{aligned}
\right.
\end{equation}

Similarly, if $V\in C^{1,3,2}([0,T]\times\mathbb{R}^2)$ and $\frac{\partial^2V}{\partial t\partial x_1},\frac{\partial^2V}{\partial x\partial x_1}$ are continuous, then
\begin{equation*}
\begin{aligned}
 &\frac{\partial}{\partial x_1}\bigg\{-\frac{\partial V}{\partial t}(t,x,x_1)+G\Big(t,x,x_1,x_2,u^*(t),-V(t,x,x_1),\\
 &\qquad-\frac{\partial V}{\partial x}(t,x,x_1),-\frac{\partial^2V}{\partial x^2}(t,x,x_1),
  -\frac{\partial V}{\partial x_1}(t,x,x_1)\Big)\bigg\}\bigg|_{x_1=X_1^*(t)}=0,\ t\in[s,T].
\end{aligned}
\end{equation*}
This is equivalent to
\begin{equation*}
\begin{aligned}
 &\quad-\frac{\partial^2V}{\partial t\partial x_1}(t,X^*(t),X_1^*(t))-\frac{\partial^2V}{\partial x\partial x_1}(t,X^*(t),X_1^*(t))b^*(t)\\
 &\quad-\frac{\partial V}{\partial x}(t,X^*(t),X_1^*(t))\frac{\partial b^*}{\partial x_1}(t)
  +\lambda\frac{\partial V}{\partial x_1}(t,X^*(t),X_1^*(t))\\
 &\quad-\frac{1}{2}\sigma^{*2}(t)\frac{\partial^3V}{\partial x^2\partial x_1}(t,X^*(t),X_1^*(t))
  -\frac{\partial\sigma^*}{\partial x_1}(t)\frac{\partial^2V}{\partial x^2}(t,X^*(t),X_1^*(t))\sigma^*(t)\\
 &\quad-\big[X^*(t)-\lambda X_1^*(t)-e^{-\lambda\delta}X_2^*(t)\big]\frac{\partial^2V}{\partial x_1^2}(t,X^*(t),X_1^*(t))\\
 &\quad+\frac{\partial f^*}{\partial x_1}(t)-\frac{\partial f^*}{\partial y}(t)\frac{\partial V}{\partial x_1}(t,X^*(t),X_1^*(t))\\
 &\quad-\frac{\partial f^*}{\partial z}(t)\bigg[\frac{\partial^2V}{\partial x\partial x_1}(t,X^*(t),X_1^*(t))\sigma^*(t)
  +\frac{\partial V}{\partial x}(t,X^*(t),X_1^*(t))\frac{\partial\sigma^*}{\partial x_1}(t)\bigg]=0,
\end{aligned}
\end{equation*}
$\forall t\in[s,T]$. On the other hand, applying (\ref{Ito's formula}) to $\frac{\partial V}{\partial x_1}(t,X^*(t),X_1^*(t))$, we get
\begin{equation*}
\begin{aligned}
&d\frac{\partial V}{\partial x_1}(t,X^*(t),X_1^*(t))=\bigg\{\frac{\partial^2V}{\partial t\partial x_1}(t,X^*(t),X_1^*(t))
 +\frac{\partial^2V}{\partial x\partial x_1}(t,X^*(t),X_1^*(t))b^*(t)\\
&\quad+\frac{1}{2}\sigma^{*2}(t)\frac{\partial^3V}{\partial x^2\partial x_1}(t,X^*(t),X_1^*(t))+\big[X^*(t)-\lambda X_1^*(t)-e^{-\lambda\delta}X_2^*(t)\big]
\end{aligned}
\end{equation*}
\begin{equation*}
\begin{aligned}
&\quad\times\frac{\partial^2V}{\partial x_1^2}(t,X^*(t),X_1^*(t))\bigg\}dt+\frac{\partial^2V}{\partial x\partial x_1}(t,X^*(t),X_1^*(t))\sigma^*(t)dW(t)\\
&=\bigg\{-\frac{\partial V}{\partial x}(t,X^*(t),X_1^*(t))\frac{\partial b^*}{\partial x_1}(t)
 -\frac{\partial\sigma^*}{\partial x_1}(t)\frac{\partial^2V}{\partial x^2}(t,X^*(t),X_1^*(t))\sigma^*(t)\\
&\qquad+\lambda\frac{\partial V}{\partial x_1}(t,X^*(t),X_1^*(t))+\frac{\partial f^*}{\partial x_1}(t)
 -\frac{\partial f^*}{\partial y}(t)\frac{\partial V}{\partial x_1}(t,X^*(t),X_1^*(t))\\
&\qquad-\frac{\partial f^*}{\partial z}(t)\bigg[\frac{\partial^2V}{\partial x\partial x_1}(t,X^*(t),X_1^*(t))\sigma^*(t)
  +\frac{\partial V}{\partial x}(t,X^*(t),X_1^*(t))\frac{\partial\sigma^*}{\partial x_1}(t)\bigg]\bigg\}dt\\
&\quad+\frac{\partial^2V}{\partial x\partial x_1}(t,X^*(t),X_1^*(t))\sigma^*(t)dW(t).
\end{aligned}
\end{equation*}
Applying again (\ref{Ito's formula}) to $\frac{\partial V}{\partial x_1}(t,X^*(t),X_1^*(t))q(t)$, noting that
\begin{equation*}
\frac{\partial V}{\partial x_1}(T,X^*(T),X_1^*(T))=-\frac{\partial\phi}{\partial x_1}(X^*(T),X_1^*(T)),
\end{equation*}
we have
\begin{equation*}
\begin{aligned}
&d\bigg\{\frac{\partial V}{\partial x_1}(t,X^*(t),X_1^*(t))q(t)\bigg\}
 =\bigg\{-\frac{\partial b^*}{\partial x_1}(t)\frac{\partial V}{\partial x}(t,X^*(t),X_1^*(t))q(t)\\
&\qquad-\frac{\partial\sigma^*}{\partial x_1}(t)\bigg[\frac{\partial^2V}{\partial x^2}(t,X^*(t),X_1^*(t))\sigma^*(t)
 +\frac{\partial f^*}{\partial z}(t)\frac{\partial V}{\partial x}(t,X^*(t),X_1^*(t))\bigg]q(t)\\
&\qquad+\lambda\frac{\partial V}{\partial x_1}(t,X^*(t),X_1^*(t))q(t)+\frac{\partial f^*}{\partial x_1}(t)q(t)\bigg\}dt\\
&\quad+\bigg[\frac{\partial^2V}{\partial x\partial x_1}(t,X^*(t),X_1^*(t))\sigma^*(t)
 +\frac{\partial f^*}{\partial z}(t)\frac{\partial V}{\partial x_1}(t,X^*(t),X_1^*(t))\bigg]q(t)dW(t).
\end{aligned}
\end{equation*}
Hence, by the uniqueness of the solution to the $p_2(t)$ part of adjoint equation (\ref{adjoint equation}), we have
\begin{equation}\label{relation of p2}
\left\{
\begin{aligned}
p_2(t)&=\frac{\partial V}{\partial x_1}(t,X^*(t),X_1^*(t))q(t),\quad\forall t\in[s,T],a.s.,\\
k_2(t)&=\bigg[\frac{\partial^2V}{\partial x\partial x_1}(t,X^*(t),X_1^*(t))\sigma^*(t)
       +\frac{\partial V}{\partial x_1}(t,X^*(t),X_1^*(t))\frac{\partial f^*}{\partial z}(t)\Big)\bigg]q(t),\\
      &\hspace{55mm}a.e.t\in[s,T],a.s.
\end{aligned}
\right.
\end{equation}
And finally
\begin{equation*}
  q(t)=\exp\bigg\{\int_s^t\frac{\partial f^*}{\partial y}(r)dW(r)-\frac{1}{2}\int_s^t\Big|\frac{\partial f^*}{\partial z}(r)\Big|^2dr\bigg\},\quad\forall t\in[s,T],a.s.
\end{equation*}
can be easily obtained by solving the forward equation of $q(t)$ directly. The proof is complete.
\end{proof}

\section{Application to Consumption and Portfolio Optimization with Recursive Utility}

In this section, we discuss a consumption and portfolio optimization problem with recursive utility in the financial market. The financial framework in this problem is initiated introduced by Chang et al. \cite{CPY11}, with classical cost functional. In this paper, we generalize their model to the case with recursive utility. The optimal portfolio and consumption strategies are obtained by both dynamic programming and maximum principle approaches, in the meanwhile the relations we obtained in Theorem 3.2 are illustrated.

Let us first describe the environment of the financial market. Consider an investor who can invest his money into a risky asset and a riskless
asset. The risky asset can be a stock, a mutual fund, etc. The riskless asset earns a fixed interest rate $r>0$.
We can treat the money invested on the riskless asset as money deposited into a bank account. We assume that
the investor can consume his/her wealth.

Let $U(t)$ be the amount invested on the risky asset and $V(t)$ is the amount invested on the riskless asset. The
total wealth is given by $X(t)=U(t)+V(t)$. We consider the situation in which the performance of the risky
asset has some memory (delay). Because many investors will look at an asset's past performance before they
invest their money on the asset, the increasing investment performance of their wealth in the past tends to drive
the investors to invest more on the risky asset, hence it can push the price of the risky asset even higher. On
the other hand, if the price has been decreasing a lot, investors tend to sell the asset and invest on other assets,
which will drive the price to go down further. To describe this phenomenon, we assume that the performance
of the risky asset depends on the following delay variables $X_1(t)$ and $X_2(t)$:
\begin{equation}\label{definition of delay terms 2}
\begin{aligned}
X_1(t)=\int_{-\delta}^0e^{\lambda\tau}X(t+\tau)d\tau,\quad
X_2(t)=X(t-\delta),\ t\in[s,T],
\end{aligned}
\end{equation}
for any initial time $s\in[0,T)$. Here $\lambda$ is a constant and $\delta>0$ is the delay parameter.
The parameter $\delta$ gives us the duration of the past that the investor usually cares about.

Let $\{W(t),t\geq0\}$ be a one-dimensional standard Brownian motion defined on a probability space $(\Omega,\mathcal{F},\mathbb{P})$.
We assume that the filtration $\mathcal{F}^0_t=\sigma\{W(\tau);0\leq\tau\leq t\}$ is augmented by
all the $\mathbb{P}$-null sets in $\mathcal{F}$.

We assume that $U(t)$ and $V(t)$ follow the stochastic differential equations:
\begin{equation}\label{price of risky asset}
dU(t)=\big[\mu_0U(t)+\mu_1X_1(t)+\mu_2X_2(t)\big]dt+\sigma U(t)dW(t),
\end{equation}
\begin{equation}\label{price of riskless asset}
\hspace{-4.65cm}dV(t)=\big[rV(t)-C(t)\big]dt,
\end{equation}
where $\mu_0,\mu_1,\mu_2$ and $\sigma$ are positive constants, and $C(t)$ is the consumption rate.

Add them together, and use the fact that $X(t)=U(t)+V(t)$, then we get the equation for the wealth $X(t)$:
\begin{equation}\label{wealth equation-1}
\left\{
\begin{aligned}
dX(t)=&\ \big[\mu_0U(t)+\mu_1X_1(t)+\mu_2X_2(t)+rV(t)-C(t)\big]dt\\
      &\ +\sigma U(t)dW(t),\ t\in[0,T],\\
 X(t)=&\ \varphi(t),\ t\in[-\delta,0],
\end{aligned}
\right.
\end{equation}
where continuous function $\varphi:[-\delta,0]\rightarrow\mathbb{R}$ is the initial condition for information about $X(t)$ for $t\in[-\delta,0]$.

Further, instead of $U(t)$ and $C(t)$, we use $c(t)\equiv C(t)/X(t)$ and $u(t)\equiv U(t)/X(t)$ as our consumption and portfolio control, respectively (note that $X(t)>0$, a.s. is proved in Lemma 2.2 of \cite{CPY11}). It is easy to see that $V(t)=X(t)-U(t)=X(t)(1-u(t))$. Now we can rewrite the equation for $X(t)$ as
\begin{equation}\label{wealth equation-2}
\left\{
\begin{aligned}
dX(t)=&\ \big[((\mu_0-r)u(t)-c(t)+r)X(t)+\mu_1X_1(t)+\mu_2X_2(t)\big]dt\\
      &\ +\sigma u(t)X(t)dW(t),\ t\in[s,T],\\
 X(t)=&\ \varphi(t)>0,\ t\in[-\delta,0].
\end{aligned}
\right.
\end{equation}

Now we define the admissible control space $\Pi$ for the control variables $u(t)$ and $c(t)$.

\begin{definition}
(Admissible Control Space) A control strategy $(u(t),c(t))$ is said to be in the admissible
control space $\Pi$ if it satisfies the following conditions:
\begin{equation*}
\begin{aligned}
&(i)\quad (u(t),c(t))\mbox{ is }\mathcal{F}_t\mbox{-adapted processes};\\
&(ii)\quad c(t)\geq0,\forall t\in[0,T];\\
&(iii)\quad \mbox{For any }t\in[0,T], \mbox{ we have }\\
&\qquad\quad|u(t)X(t)|\leq\Lambda_1|X(t)+\mu_2X_1(t)|,\quad|c(t)X(t)|\leq\Lambda_2|X(t)+\mu_2X_1(t)|,
\end{aligned}
\end{equation*}
where $\Lambda_1,\Lambda_2>0$ are positive constants.
\end{definition}

\begin{remark}
The condition (iii) is sufficient to obtain the result in Lemma 2.2 of \cite{CPY11}.
\end{remark}

The investor wants to minimize the following recursive utility
\begin{equation}\label{recursive utility functional example}
J(u(\cdot),c(\cdot)):=-Y(t)\big|_{t=0},
\end{equation}
over the admissible control space $\Pi$, where
\begin{equation}\label{controlled BSDDE example}
\left\{
\begin{aligned}
-dY(t)&=\big[-\beta Y(t)+\frac{1}{\gamma}\big(c(t)X(t)\big)^\gamma\big]dt-Z(t)dW(t),\quad t\in[0,T],\\
  Y(T)&=\frac{1}{\gamma}\big(X(T)+\theta X_1(T)\big)^\gamma.
\end{aligned}
\right.
\end{equation}
and $\beta\geq0,\gamma\in(-\infty,1),\gamma\neq0,\theta\in\mathbb{R}$ are constants.

\begin{remark}
The recursive utility functional defined in (\ref{controlled BSDDE example}) with generator
\begin{equation}\label{generator of recursive utility example}
f(t,x,x_1,y,z,u,c)=-\beta y+\frac{1}{\gamma}(cx)^\gamma,
\end{equation}
stands for some standard additive utility of recursive type. It can be easily checked that $f$ defined above is concave with respect to $(c,y)$ and increasing with respect to $c$, which are classical properties that utility functions must satisfy. Recursive utility such as (\ref{generator of recursive utility example}) is meaningful and nontrivial generalization of the classical additive utility and has many applications in mathematical economics and mathematical finance. For more details about recursive utilities, see \cite{SS99}, \cite{EPQ97,EPQ01} and the references therein.
\end{remark}

Since we are going to involve the dynamic programming principle, we will adopt the formulation as in Section 2. For given $(s,\varphi)\in[0,T)\times C([-\delta,0];\mathbb{R})$ and admissible control $(u(\cdot),c(\cdot))$, the wealth equation is
\begin{equation}\label{wealth equation example DPP}
\left\{
\begin{aligned}
dX^{s,\varphi;u,c}(t)=&\ \Big\{\big[(\mu_0-r)u(t)-c(t)+r\big]X^{s,\varphi;u,c}(t)+\mu_1X_1^{s,\varphi;u,c}(t)\\
                      &\quad+\mu_2X_2^{s,\varphi;u,c}(t)\Big\}dt+\sigma u(t)X^{s,\varphi;u,c}(t)dW(t),\ t\in[s,T],\\
 X^{s,\varphi;u,c}(t)=&\ \varphi(t)>0,\ t\in[-\delta,0].
\end{aligned}
\right.
\end{equation}
and the recursive utility functional is defined as
\begin{equation}\label{recursive utility functional example DPP}
\begin{aligned}
J(s,\varphi;u(\cdot),c(\cdot))=-Y^{s,\varphi;u,c}(t)\big|_{t=s},
\end{aligned}
\end{equation}
where
\begin{equation}\label{controlled BSDDE example DPP}
\left\{
\begin{aligned}
-dY^{s,\varphi;u,c}(t)&=\Big\{-\beta Y^{s,\varphi;u,c}(t)+\frac{1}{\gamma}\big(c(t)X^{s,\varphi;u,c}(t)\big)^\gamma\Big\}dt\\
                      &\quad-Z^{s,\varphi;u,c}(t)dW(t),\ t\in[s,T],\\
  Y^{s,\varphi;u,c}(T)&=\frac{1}{\gamma}\big(X^{s,\varphi;u,c}(T)+\theta X_1^{s,\varphi;u,c}(T)\big)^\gamma.
\end{aligned}
\right.
\end{equation}

This problem can be reformulated as follows. The state process $(X^{s,\varphi;u,c}(\cdot),\\Y^{s,\varphi;u,c}(\cdot),Z^{s,\varphi;u,c}(\cdot))$ of our
system is described by the following coupled FBSDDE
\begin{equation}\label{controlled FBSDDE example}
\left\{
\begin{aligned}
 dX^{s,\varphi;u,c}(t)=&\ \Big\{\big[(\mu_0-r)u(t)-c(t)+r\big]X^{s,\varphi;u,c}(t)+\mu_1X_1^{s,\varphi;u,c}(t)\\
                       &\quad+\mu_2X_2^{s,\varphi;u,c}(t)\Big\}dt+\sigma u(t)X^{s,\varphi;u,c}(t)dW(t),\\
-dY^{s,\varphi;u,c}(t)=&\ \Big\{-\beta Y^{s,\varphi;u,c}(t)+\frac{1}{\gamma}\big(c(t)X^{s,\varphi;u,c}(t)\big)^\gamma\Big\}dt\\
                       &\ -Z^{s,\varphi;u,c}(t)dW(t),\quad t\in[s,T],\\
  X^{s,\varphi;u,c}(t)=&\ \varphi(t)>0,\ t\in[s-\delta,s],\\
  Y^{s,\varphi;u,c}(T)=&\ \frac{1}{\gamma}\big(X^{s,\varphi;u,c}(T)+\theta X_1^{s,\varphi;u,c}(T)\big)^\gamma,
\end{aligned}
\right.
\end{equation}
and the cost functional is given of the form
\begin{equation}\label{cost functional for FBSDDE example}
\begin{aligned}
 &J(s,\varphi;u(\cdot),c(\cdot))=-Y^{s,\varphi;u,c}(s)\\
=&\ \mathbb{E}^{s,\varphi;u,c}\bigg\{\int_s^T\Big[-\beta Y^{s,\varphi;u,c}(t)+\frac{1}{\gamma}\big(c(t)X^{s,\varphi;u,c}(t)\big)^\gamma\Big]dt\\
 &\qquad\qquad+\frac{1}{\gamma}\big(X^{s,\varphi;u,c}(T)+\theta X_1^{s,\varphi;u,c}(T)\big)^\gamma\bigg\}.
\end{aligned}
\end{equation}

The consumption and portfolio optimization problem is to find an admissible $(u^*(\cdot),c^*(\cdot))$ such that
\begin{equation}\label{value function for recursive utility example}
V(s,\varphi)=J(s,\varphi;u^*(\cdot),c^*(\cdot))=\essinf\limits_{(u(\cdot),c(\cdot))\in\Pi}J(s,\varphi;u(\cdot),c(\cdot)),
\end{equation}
for all $(s,\varphi)\in[0,T)\times C([-\delta,0];\mathbb{R})$.

\subsection{Dynamic Programming Approach}

In this subsection, we solve the above consumption and portfolio optimization problem, applying Bellman's dynamic programming approach.

In this case, if the value function $V$ only depends on $\varphi$ through $(x,x_1)$, i.e.,
\begin{equation}\label{value function example-DPP}
V(s,\varphi)=V(s,x,x_1,x_2)=V(s,x,x_1),
\end{equation}
where $V:[0,T]\times\mathbb{R}^2\rightarrow\mathbb{R}$ with
\begin{equation*}
x=x(\varphi)=\varphi(0),\quad x_1=x_1(\varphi)=\int_{-\delta}^0e^{\lambda\tau}\varphi(\tau)d\tau,
\end{equation*}
then the generalized HJB equation (\ref{generalized HJB equation}) that $V(s,x,x_1)$ should satisfy, reduces to
\begin{equation}\label{HJB equation example}
\left\{
\begin{aligned}
&0=\frac{\partial V}{\partial s}(s,x,x_1)+\max\limits_{u}\bigg\{\frac{1}{2}\sigma^2u^2x^2\frac{\partial^2 V}{\partial x^2}(s,x,x_1)
 +(\mu_0-r)ux\frac{\partial V}{\partial x}(s,x,x_1)\bigg\}\\
&\qquad+\max\limits_{c\geq0}\bigg\{-cx\frac{\partial V}{\partial x}(s,x,x_1)+\frac{1}{\gamma}c^\gamma x^\gamma\bigg\}
 +(rx+\mu_1x_1+\mu_2x_2)\frac{\partial V}{\partial x}(s,x,x_1)\\
&\qquad+(x-\lambda x_1-e^{-\lambda \delta}x_2)\frac{\partial V}{\partial x_1}(s,x,x_1)-\beta V(s,x,x_1),\ \forall x_2\in\mathbb{R},\\
&V(T,x,x_2)=-\frac{1}{\gamma}\big(x+\theta x_1\big)^\gamma.
\end{aligned}
\right.
\end{equation}

By (\ref{HJB equation example}), the candidate for the optimal consumption and portfolio strategies is
\begin{equation*}
\begin{aligned}
u^*(s)\equiv u^*(s,x,x_1)&=\ -\frac{(\mu_0-r)\frac{\partial V}{\partial x}(s,x,x_1)}{\sigma^2x\frac{\partial^2 V}{\partial x^2}(s,x,x_1)},\\
c^*(s)\equiv c^*(s,x,x_1)&=\ \frac{\big[\frac{\partial V}{\partial x}(s,x,x_1)\big]^{\frac{1}{\gamma-1}}}{x}.
\end{aligned}
\end{equation*}
Plug these into (\ref{HJB equation example}), we have
\begin{equation}\label{HJB equation example-by optimal control}
\begin{aligned}
 &\ \beta V(s,x,x_1)-\frac{\partial V}{\partial s}(s,x,x_1)=-\frac{1}{2}\frac{(\mu_0-r)^2\big[\frac{\partial V}{\partial x}(s,x,x_1)\big]^2}{\sigma^2\frac{\partial^2 V}{\partial x^2}(s,x,x_1)}\\
 &\quad+\Big(\frac{1}{\gamma}-1\Big)\Big[\frac{\partial V}{\partial x}(s,x,x_1)\Big]^{\frac{\gamma}{\gamma-1}}
  +(rx+\mu_1x_1)\frac{\partial V}{\partial x}(s,x,x_1)\\
 &\quad+(x-\lambda x_1)\frac{\partial V}{\partial x_1}(s,x,x_1)
  +\bigg[\mu_2\frac{\partial V}{\partial x}(s,x,x_1)-e^{-\lambda \delta}\frac{\partial V}{\partial x_1}(s,x,x_1)\bigg]x_2,\quad\forall x_2\in\mathbb{R},\\
\end{aligned}
\end{equation}

Motivated by the terminal condition in (\ref{HJB equation example}), we try a value of the form
\begin{equation}\label{value function example}
V(s,x,x_1)=-\frac{1}{\gamma}Q(s)\big(x+\theta x_1\big)^\gamma,
\end{equation}
where $Q(\cdot)$ is some differentiable deterministic function with $Q(T)=1$. We then have
\begin{equation*}
\begin{aligned}
   \frac{\partial V}{\partial s}(s,x,x_1)&=-\frac{1}{\gamma}Q'(s)\big(x+\theta x_1\big)^\gamma,\
   \frac{\partial V}{\partial x}(s,x,x_1)=-Q(s)\big(x+\theta x_1\big)^{\gamma-1},\\
\frac{\partial^2V}{\partial x^2}(s,x,x_1)&=-(\gamma-1)Q(s)\big(x+\theta x_1\big)^{\gamma-2},\
 \frac{\partial V}{\partial x_1}(s,x,x_1)=-\theta Q(s)\big(x+\theta x_1\big)^{\gamma-1}.
\end{aligned}
\end{equation*}
Put them into (\ref{HJB equation example-by optimal control}), we obtain that
\begin{equation}\label{HJB equation example-by optimal control-2}
\begin{aligned}
 &\ \frac{1}{\gamma}\Big[\beta Q(s)-Q'(s)\Big]\big(x+\theta x_1\big)^{\gamma}\\
=&\ -\frac{(\mu_0-r)^2}{2\sigma^2(\gamma-1)}Q(s)\big(x+\theta x_1\big)^{\gamma}
  +\Big(\frac{1}{\gamma}-1\Big)\big[Q(s)\big]^{\frac{\gamma}{\gamma-1}}\big(x+\theta x_1\big)^{\gamma}\\
 &\ +\Big\{(r+\theta)x+(\mu_1-\lambda\theta)x_1
  +\big(\mu_2-e^{-\lambda \delta}\theta\big)x_2\Big\}Q(s)\big(x+\theta x_1\big)^{\gamma-1},\ \forall x_2\in\mathbb{R}.
\end{aligned}
\end{equation}

When will this equation admit a classical solution independent of $x_2$? Note that the controlled FBSDDE is of the form (\ref{controlled FBSDDE-finite dimensional}), that is, the condition (\ref{assumption of b,sigma,f}) is satisfied. From the condition (\ref{system of first order PDEs}), if
\begin{equation}\label{condition 1-DPP}
\hspace{-2cm}\theta=\mu_2e^{\lambda\delta},\ \mu_2>0,
\end{equation}
\begin{equation}\label{condition 2-DPP}
\mu_1-\lambda\mu_2e^{\lambda\delta}=(r+\mu_2e^{\lambda\delta})\mu_2e^{\lambda\delta},
\end{equation}
holds, we can rewrite (\ref{HJB equation example-by optimal control-2}) as
\begin{equation}\label{HJB equation example-by optimal control-3}
\begin{aligned}
 &\ \frac{1}{\gamma}\Big[\beta Q(s)-Q'(s)\Big]\big(x+\mu_2e^{\lambda\delta}x_1\big)^{\gamma}
=-\frac{(\mu_0-r)^2}{2\sigma^2(\gamma-1)}Q(s)\big(x+\mu_2e^{\lambda\delta}x_1\big)^{\gamma}\\
 &\qquad+\Big(\frac{1}{\gamma}-1\Big)\big[Q(s)\big]^{\frac{\gamma}{\gamma-1}}\big(x+\mu_2e^{\lambda\delta}x_1\big)^{\gamma}
  +(r+\mu_2e^{\lambda\delta})Q(s)\big(x+\mu_2e^{\lambda\delta}x_1\big)^\gamma,
\end{aligned}
\end{equation}
which is independent of $x_2$ and is an ``effective" HJB equation. Canceling the term $\big(x+\mu_2e^{\lambda\delta}x_1\big)^\gamma$ on both sides, we can get
\begin{equation}\label{equation of Q(t)-1}
\begin{aligned}
 \frac{1}{\gamma}\Big[\beta Q(s)-Q'(s)\Big]=-\frac{(\mu_0-r)^2}{2\sigma^2(\gamma-1)}Q(s)\Big(\frac{1}{\gamma}-1\Big)\big[Q(s)\big]^{\frac{\gamma}{\gamma-1}}+(r+\mu_2e^{\lambda\delta})Q(s).
\end{aligned}
\end{equation}
That is
\begin{equation}\label{equation of Q(t)}
\left\{
\begin{aligned}
Q'(s)=&\ (\gamma-1)\big[Q(s)\big]^{\frac{\gamma}{\gamma-1}}+\Delta Q(s),\ s\in[0,T],\\
 Q(T)=&\ 1,
\end{aligned}
\right.
\end{equation}
where
\begin{equation}\label{definition of Lambda}
\begin{aligned}
\Delta\equiv\beta+\frac{(\mu_1-r)^2\gamma}{2\sigma^2(\gamma-1)}-\gamma(r+\mu_3e^{\lambda\delta})>0.
\end{aligned}
\end{equation}
By some elementary technique of solving ODEs, we can obtain the following explicit solution
\begin{equation}\label{explicit solution of Q(t)}
\begin{aligned}
Q(s)=&\ \left[\left(1-\frac{1-\gamma}{\Delta}\right)e^{\frac{\Delta(T-s)}{1-\gamma}}+\frac{1-\gamma}{\Delta}\right]^{1-\gamma},\ s\in[0,T].
\end{aligned}
\end{equation}

We have proved the following result.

\begin{theorem}
Assume that (\ref{condition 1-DPP}) and (\ref{condition 2-DPP}) hold, then the function $V(s,x,x_1)$ given by (\ref{value function example}) is a classical solution to the generalized HJB equation (\ref{HJB equation example}), and it is equal to the value function defined by (\ref{value function example-DPP}). In addition, the optimal portfolio and consumption strategies are given by
\begin{equation}\label{optimal portfolio proportion}
\hspace{-2cm}u^*(s)\equiv u^*(s,X^*(s),X_1^*(s))=\frac{(\mu_0-r)\big(X^*(s)+\mu_2e^{\lambda\delta}X_1^*(s)\big)}{(1-\gamma)\sigma^2X^*(s)},
\end{equation}
\begin{equation}\label{optimal consumption rate proportion}
c^*(s)\equiv c^*(s,X^*(s),X_1^*(s))=\frac{X^*(s)+\mu_2e^{\lambda\delta}X_1^*(s)}{X^*(s)}Q(s)^{\frac{1}{\gamma-1}},\ s\in[0,T],
\end{equation}
where $(X^*(\cdot),X_1^*(\cdot))$ is the corresponding solution to SDDE (\ref{wealth equation-2}) with (\ref{definition of delay terms 2}), and $Q(\cdot)$ satisfies the ODE (\ref{equation of Q(t)}), which admits the explicit solution (\ref{explicit solution of Q(t)}).
\end{theorem}

\subsection{Maximum Principle Approach}

In this subsection, we derive the optimal portfolio and consumption strategies (\ref{optimal portfolio proportion}) and (\ref{optimal consumption rate proportion}), applying Pontryagin's maximum principle approach.

By (\ref{Hamiltonian function}), the Hamiltonian function takes the form
\begin{equation}\label{Hamiltonian function example}
\begin{aligned}
&H(t,x,x_1,x_2,y,z,u,p_1,p_2,q,k_1)=p_1\Big\{\big[(\mu_0-r)u-c+r\big]x+\mu_1x_1+\mu_2x_2\Big\}\\
&\qquad+p_2\big(x-\lambda x_1-e^{-\lambda\delta}x_2\big)+k_1\sigma ux-q\big[-\beta y+\frac{1}{\gamma}c^\gamma x^\gamma\big].
\end{aligned}
\end{equation}
For candidate optimal control strategies $(u^*(\cdot),c^*(\cdot))$ and the corresponding state $(X^*(\cdot),Y^*(\cdot),\\Z^*(\cdot))$, the adjoint equation (\ref{adjoint equation}) reduces to
\begin{equation}\label{adjoint equations example}
\left\{
\begin{aligned}
-dp_1(t)&=\bigg\{\big[(\mu_0-r)u^*(t)-c^*(t)+r\big]p_1(t)+p_2(t)+\sigma u^*(t)k_1(t)\\
        &\qquad-q(t)c^*(t)^\gamma X^*(t)^{\gamma-1}\bigg\}dt-k_1(t)dW(t),\\
-dp_2(t)&=\big(\mu_1p_1(t)-\lambda p_2(t)\big)dt-k_2(t)dW(t),\\
-dp_3(t)&=\big(\mu_2p_1(t)-e^{-\lambda\delta}p_2(t)\big)dt,\ q(t)=-\beta q(t)dt,\quad t\in[0,T],\\
    q(0)&=1,\ p_1(T)=-\big(X^*(T)+\theta X_1^*(T)\big)^{\gamma-1}q(T),\\
  p_2(T)&=-\theta\big(X^*(T)+\theta X_1^*(T)\big)^{\gamma-1}q(T),\ p_3(T)=0.
\end{aligned}
\right.
\end{equation}

The maximum condition (\ref{maximum condition}) in Theorems 3.1 tell us that, we can find $u^*(\cdot)$ and $c^*(\cdot)$ by maximizing
\begin{equation*}
u\rightarrow H(t,x,x_1,x_2,y,z,u,c,p_1,p_2,q,k_1)
\end{equation*}
and
\begin{equation*}
c\rightarrow H(t,x,x_1,x_2,y,z,u,c,p_1,p_2,q,k_1)
\end{equation*}
over all $u$ and $c\geq0$, respectively. Then we have
\begin{equation}\label{maximum condition in k}
\begin{aligned}
 &\frac{\partial H}{\partial u}(t,x,x_1,x_2,y,z,u,c,p_1,p_2,q,k_1)\Big|_{u=u^*(t)}\\
 &=(\mu_0-r)p_1(t)X^*(t)+\sigma k_1(t)X^*(t)=0,
\end{aligned}
\end{equation}
and
\begin{equation}\label{maximum condition in c}
\begin{aligned}
 &\frac{\partial H}{\partial c}(t,x,x_1,x_2,y,z,u,c,p_1,p_2,q,k_1)\Big|_{c=c^*(t)}\\
 &=-p_1(t)X^*(t)-q(t)c^*(t)^{\gamma-1}X^*(t)^\gamma=0.
\end{aligned}
\end{equation}

By (\ref{maximum condition in c}), we directly achieve that
\begin{equation}\label{optimal consumption rate proportion with p1(t)}
c^*(t)=\frac{1}{X^*(t)}\bigg\{-\frac{p_1(t)}{q(t)}\bigg\}^{\frac{1}{\gamma-1}},\quad t\in[0,T].
\end{equation}

Motivated by the terminal condition for $p_1(t)$ in (\ref{adjoint equations example}), we try to find $p_1(t)$ of the form
\begin{equation}\label{try form of p1(t)}
p_1(t)=-Q(t)\big(X^*(t)+\theta X_1^*(t)\big)^{\gamma-1}q(t),
\end{equation}
where $Q(\cdot)$ is a deterministic differentiable function with $Q(T)=1$.

Applying (\ref{Ito's formula}) to (\ref{try form of p1(t)}), we can get (noting (\ref{applying Ito's formula to X-1}))
\begin{equation}\label{dp1(t)}
\begin{aligned}
dp_1(t)=&\ -Q'(t)q(t)\big(X^*(t)+\theta X_1^*(t)\big)^{\gamma-1}dt+\beta Q(t)q(t)\big(X^*(t)+\theta X_1^*(t)\big)^{\gamma-1}dt\\
        &\ -Q(t)q(t)(\gamma-1)\big(X^*(t)+\theta X_1^*(t)\big)^{\gamma-2}\big(dX^*(t)+\theta dX_1^*(t)\big)\\
        &\ -\frac{1}{2}Q(t)q(t)(\gamma-1)(\gamma-2)\big(X^*(t)+\theta X_1^*(t)\big)^{\gamma-3}\big(dX^*(t)+\theta dX_1^*(t)\big)^2\\
       =&\ \bigg\{-\big(X^*(t)+\theta X_1^*(t)\big)^{\gamma-1}Q'(t)q(t)-\big(X^*(t)+\theta X_1^*(t)\big)^{\gamma-2}(\gamma-1)Q(t)q(t)\\
        &\quad\times\Big[\big((\mu_0-r)u^*(t)-c^*(t)+r\big)X^*(t)+\mu_1X_1^*(t)+\mu_2X_2^*(t)\\
        &\qquad\ +\theta\big[X^*(t)-\lambda X_1^*(t)-e^{-\lambda\delta}X_2^*(t)\big]\Big]\\
        &\quad-\frac{1}{2}\big(X^*(t)+\theta X_1^*(t)\big)^{\gamma-3}(\gamma-1)(\gamma-2)\sigma^2Q(t)q(t)u^{*2}(t)X^{*2}(t)\\
        &\quad+\big(X^*(t)+\theta X_1^*(t)\big)^{\gamma-1}\beta Q(t)q(t)dt\bigg\}dt\\
        &\ -\big(X^*(t)+\theta X_1^*(t)\big)^{\gamma-2}(\gamma-1)\sigma Q(t)q(t)u^*(t)X^*(t)dW(t).
\end{aligned}
\end{equation}
Comparing the coefficient of $dW(t)$ in (\ref{dp1(t)}) with that in (\ref{adjoint equations example}), we obtain
\begin{equation}\label{q1(t)}
k_1(t)=-\big(X^*(t)+\theta X_1^*(t)\big)^{\gamma-2}(\gamma-1)\sigma Q(t)q(t)u^*(t)X^*(t),\quad t\in[0,T], a.s.
\end{equation}

Next, by the adjoint equation (\ref{adjoint equations example}), the condition (\ref{p3=0}) that $p_3(t)=0, \forall t\in[0,T]$ can be formulated as follows:
\begin{equation}\label{p3=0's consequence}
\mu_2p_1(t)-e^{-\lambda\delta}p_2(t)=0,\ \forall t\in[0,T].
\end{equation}
That is
\begin{equation}\label{try form of p2(t)}
p_2(t)=\mu_2e^{\lambda\delta}p_1(t)=-\mu_2e^{\lambda\delta}Q(t)\big(X^*(t)+\theta X_1^*(t)\big)^{\gamma-1}q(t),\ \forall t\in[0,T].
\end{equation}
By the terminal conditions in (\ref{adjoint equations example}), we must have
\begin{equation}\label{condition 1-MP}
\theta=\mu_2e^{\lambda\delta}.
\end{equation}
That is
\begin{equation}\label{p1(t)-final}
p_1(t)=-Q(t)\big(X^*(t)+\mu_2e^{\lambda\delta}X_1^*(t)\big)^{\gamma-1}q(t),
\end{equation}
and
\begin{equation}\label{p2(t)-final}
p_2(t)=\mu_2e^{\lambda\delta}p_1(t)=-\mu_2e^{\lambda\delta}Q(t)\big(X^*(t)+\mu_2e^{\lambda\delta}X_1^*(t)\big)^{\gamma-1},\ \forall t\in[0,T].
\end{equation}
Consequently, we obtain
\begin{equation}\label{q1(t)-final}
k_1(t)=-\big(X^*(t)+\mu_2e^{\lambda\delta}X_1^*(t)\big)^{\gamma-2}(\gamma-1)\sigma Q(t)q(t)u^*(t)X^*(t),
\end{equation}
and
\begin{equation}\label{q2(t)-final}
\begin{aligned}
 &k_2(t)=-\mu_2e^{\lambda\delta}k_1(t)\\
=&\ -\mu_2e^{\lambda\delta}\big(X^*(t)+\mu_2e^{\lambda\delta}X_1^*(t)\big)^{\gamma-2}(\gamma-1)\sigma Q(t)q(t)u^*(t)X^*(t),\ \forall t\in[0,T], a.s.
\end{aligned}
\end{equation}
Putting (\ref{p1(t)-final}) and (\ref{q1(t)-final}) into (\ref{maximum condition in k}) and (\ref{optimal consumption rate proportion with p1(t)}), we can get
\begin{equation}\label{optimal portfolio proportion-1}
\hspace{-2.2cm}u^*(t)=\frac{(\mu_0-r)\big(X^*(t)+\mu_2e^{\lambda\delta}X_1^*(t)\big)}{(1-\gamma)\sigma^2X^*(t)},\
\end{equation}
\begin{equation}\label{optimal consumption rate proportion-1}
c^*(t)=\frac{X^*(t)+\mu_2e^{\lambda\delta}X_1^*(t)}{X^*(t)}\big[Q(t)\big]^{\frac{1}{\gamma-1}},\ t\in[0,T], a.s.
\end{equation}

Next, comparing the coefficient of $dt$ in (\ref{dp1(t)}) with that in (\ref{adjoint equations example}), if the following assumption
\begin{equation}\label{condition 2-MP}
\mu_1-\lambda\mu_2e^{\lambda\delta}=(r+\mu_2e^{\lambda\delta})\mu_2e^{\lambda\delta}
\end{equation}
holds, we obtain that
\begin{equation}\label{ODE of Q(t)-this paper}
\left\{
\begin{aligned}
Q'(t)=&\ \Big[\beta+\frac{(\mu_1-r)^2\gamma}{2\sigma^2(\gamma-1)}-\gamma(r+\mu_2e^{\lambda\delta})\Big]Q(t)
       +(\gamma-1)\big[Q(t)\big]^{\frac{\gamma}{\gamma-1}},\ t\in[0,T],\\
 Q(T)=&\ 1,
\end{aligned}
\right.
\end{equation}
which is exactly the same ODE (\ref{equation of Q(t)}) in Theorem 4.2, and the explicit solution to it is (\ref{explicit solution of Q(t)}).

We have proved the following result.

\begin{theorem} Suppose that (\ref{condition 1-MP}) and (\ref{condition 2-MP}) hold. Then the optimal portfolio and consumption strategies are given by
\begin{equation}\label{optimal portfolio proportion-2}
\hspace{-3cm}u^*(t)\equiv u^*(t,X^*(t),X_1^*(t))=\frac{(\mu_0-r)\big(X^*(t)+\mu_2e^{\lambda\delta}X_1^*(t)\big)}{(1-\gamma)\sigma^2X^*(t)},
\end{equation}
\begin{equation}\label{optimal consumption rate proportion-2}
c^*(t)\equiv c^*(t,X^*(t),X_1^*(t))=\frac{X^*(t)+\mu_2e^{\lambda\delta}X_1^*(t)}{X^*(t)}\big[Q(t)\big]^{\frac{1}{\gamma-1}},\ t\in[0,T], a.s.,
\end{equation}
where $(X^*(\cdot),X_1^*(\cdot))$ is the corresponding solution to SDDE (\ref{wealth equation-2}) with (\ref{definition of delay terms 2}), and $Q(\cdot)$ satisfies ODE (\ref{ODE of Q(t)-this paper}), which admits the explicit solution (\ref{explicit solution of Q(t)}).
\end{theorem}

\subsection{Relationship and Some Discussions}

In addition, the relations in Theorem 3.2 can be easily verified. In fact,
relationship (\ref{relation with respect time variable}) is obvious
from (\ref{HJB equation example}). And the relations given in
(\ref{relation with respect state variable}) can be easily obtained
from (\ref{value function example}), (\ref{p1(t)-final}), (\ref{p2(t)-final}), (\ref{q1(t)-final}), (\ref{q2(t)-final}) and $q(t)=e^{-\beta t}, \forall t\in[0,T]$.

Note that the conditions (\ref{condition 1-DPP}), (\ref{condition 2-DPP}) are the same as (\ref{condition 1-MP}), (\ref{condition 2-MP}), respectively. The conditions (\ref{condition 1-DPP}), (\ref{condition 2-DPP}) comes from the system of first order PDEs (\ref{system of first order PDEs}), and conditions (\ref{condition 1-MP}), (\ref{condition 2-MP}) relies heavily on the condition (\ref{p3=0}) about the adjoint processes. This is not by chance but the natural requirement of our problem being in a finite dimensional space.

Moreover, from the condition (\ref{condition 2-DPP}) or equivalently (\ref{condition 2-MP}), we can get
\begin{equation}
\mu_1=\lambda\mu_2e^{\lambda\delta}+(r+\mu_2e^{\lambda\delta})\mu_2e^{\lambda\delta}=\mu_2e^{\lambda\delta}(\lambda+r+\mu_2e^{\lambda\delta}).
\end{equation}
So is is easy to see that $\mu_1=0$ if and only if $\mu_2=0$, provided that $\mu_2\geq0$ and $\lim\limits_{\mu_2\rightarrow\infty}\mu_1=\infty$. In other words, the price dynamics of $X(t)$ must depend on both $X_1(t)$ and $X_2(t)$ at the same time, in order to obtain the explicit representations of $V,u^*$ and $c^*$ in a finite dimensional space. Otherwise, when $\mu_2=0$ (then $\mu_1=0$), that is, the dynamic equation of $X(t)$ does not depend on $X_1(t)$ and $X_2(t)$ explicitly, our model reduces to the consumption and portfolio optimization model with recursive utility but without time delay.

\section{Conclusion}

In this paper, we have discussed Pontryagin's maximum principle, Bellman's dynamic programming and their relationship for the stochastic recursive optimal control problems with time delay, when only the pointwise and distributed time delays in the state variable is considered. One advantage for this kind time delay is that the corresponding generalized HJB equation is finite dimensional, under some suitable conditions on the coefficients. Under the assumption that the value function is smooth enough, its relations to the adjoint processes and generalized Hamiltonian function are obtained. A consumption and portfolio optimization problem with recursive utility in the financial market, was discussed to show the applications of our result. Explicit solutions for the optimal portfolio and consumption strategy in the finite dimensional space derived by the two approaches, coincide.

Potential extensions of the present work include stochastic optimal control problems with time delay under model uncertainty (Pamen \cite{Pamen13}) and stochastic differential games (\O ksendal and Sulem \cite{OS14}) under model uncertainty. Problems with time delay in control variables (\cite{ChenWu10}, \cite{Yu12}) and time varying delay in control variables (Zhang et al. \cite{ZDX06}, Zhang et al. \cite{ZFH11}, Wang and Zhang \cite{WZ13}), are rather challenging. These will be considered in our future research.

\section*{Acknowledgments}
The content of this paper was presented by the first author on the Symposium of Mathematical Control Theory and Application for Young Researchers in China, on July 2013. He would like to thank Professor Xu Zhang for some valuable discussions and the hospitality of School of Mathematical Sciences, South China Normal University, Guangzhou. Also the content of this paper was presented by the first author on the Fourth IMS-FPS Workshop in Australia, on July 2014. He would like to thank Professor R. Carmona for some desirable discussions and the hospitality of School of Mathematical Sciences, University of Technology, Sydney.
%
%

\medskip
Received xxxx 20xx; revised xxxx 20xx.
\medskip

\end{document}